\documentclass[11pt]{article}
\usepackage[ansinew]{inputenc}
\usepackage{graphicx}
\usepackage{color}
\usepackage{times}
\usepackage[hidelinks]{hyperref}
\usepackage{enumerate,latexsym}
\usepackage{latexsym}
\usepackage{amsmath,amssymb}
\usepackage{graphicx}

\usepackage{fullpage}

\usepackage{amsmath,amsthm,amsfonts,amssymb}
\usepackage{graphicx}
\usepackage{dsfont}

\makeatletter
\def\namedlabel#1#2{\begingroup
 #2%
 \def\@currentlabel{#2}%
 \phantomsection\label{#1}\endgroup
}
\makeatother

\theoremstyle{plain}
\newtheorem*{theorem*}{Theorem}
\newtheorem*{thmex*}{Theorem~\ref{example}}
\newtheorem*{thmasymp*}{Theorem~\ref{thmAsymp}}
\newtheorem{theorem}{Theorem}[section]

\newtheorem{claim}[theorem]{Claim}

\newtheorem{corollary}[theorem]{Corollary}

\newtheorem{definition}[theorem]{Definition}

\newtheorem{lemma}[theorem]{Lemma}

\newtheorem{proposition}[theorem]{Proposition}
\newtheorem{remark}[theorem]{Remark}

\newcommand{\Gi}{{G_\infty}}
\newcommand{\de}{\delta}
\newcommand{\es}{\emptyset}
\newcommand{\R}{\mathbb{R}}

\newcommand{\Z}{\mathbb{Z}}
\newcommand{\N}{{\mathbb{N}}}
\newcommand{\sn}[1]{{\mathbb{S}^{#1}}}
\newcommand{\hn}[1]{{\mathbb{H}^{#1}}}
\newcommand{\forma}[1]{\langle #1 \rangle}
\newcommand{\soma}[2]{\displaystyle \sum_{#1}^{#2}}
\newcommand{\inte}[1]{\displaystyle \int_{#1}}

\newcommand{\G}{\Gamma}
\newcommand{\wt}{\widetilde}
\newcommand{\g}{\gamma}
\newcommand{\Int}{\mbox{\rm Int}}
\renewcommand{\a}{\alpha}

\newcommand{\cA}{{\mathcal A}}
\newcommand{\cC}{{\mathcal C}}

\newcommand{\cE}{{\mathcal E}}

\newcommand{\cH}{{\mathcal H}}

\newcommand{\cL}{{\mathcal L}}
\newcommand{\cM}{{\mathcal M}}

\newcommand{\cP}{{\mathcal P}}
\newcommand{\cR}{{\mathcal R}}
\newcommand{\cS}{{\mathcal S}}
\newcommand{\cT}{{\mathcal T}}
\newcommand{\grad}{\mbox{\rm grad}}

\newcommand{\norma}[1]{\Vert #1 \Vert}

\newcommand{\n}{\nabla}

\newcommand{\Et}{\widetilde{E}}

\newcommand{\ve}{\varepsilon}

\newcommand{\abs}[1]{\vert #1 \vert}
\newcommand{\Hess}{\mbox{\rm Hess}}
\newcommand{\len}{\mbox{\rm Length}}
\newcommand{\A}{\mbox{\rm Area}}

\newcommand{\ed}{\end{document}}
\renewcommand{\S}{\Sigma}
\newcommand{\Inj}{\mbox{\rm Inj}}
\newcommand{\ol}{\overline}
\newcommand{\wh}{\widehat}
\definecolor{rrr}{rgb}{.9,0,.1}

\definecolor{rr}{rgb}{.8,0,.3}

\definecolor{b}{rgb}{.1,.1,.7}

\definecolor{pp}{rgb}{.5,0,.7}

\definecolor{drr}{rgb}{.5,0,.4}
\definecolor{y}{rgb}{.7,.2,0}

\definecolor{plum}{rgb}{0.5,0,0.5}

\usepackage{pdfsync}

\begin{document}

\title{Properly immersed surfaces in hyperbolic 3-manifolds}

\author{
William H. Meeks III\thanks{This material is based upon
 work for the NSF under Award No. DMS -
 1309236. Any opinions, findings, and conclusions or recommendations
 expressed in this publication are those of the authors and do not
 necessarily reflect the views of the NSF.}
 \and \'Alvaro K. Ramos\thanks{Both authors were partially supported
 by CNPq - Brazil, grant no. 400966/2014-0.}
}

\maketitle

\begin{abstract}
We study complete finite topology immersed surfaces $\S$ in complete
Riemannian $3$-manifolds
$N$ with sectional curvature $K_N\leq -a^2\leq 0$, such that the absolute mean
curvature
function of $\S$ is bounded from above by $a$ and its injectivity radius
function is not
bounded away from zero on each of its annular
end representatives. We prove that such a surface
$\S$ must be proper in $N$ and its total curvature must be equal to $2\pi \chi(\S)$.
If $N$ is a hyperbolic $3$-manifold of finite volume and $\S$ is a properly immersed
surface of finite topology with nonnegative
constant mean curvature less than 1,
then we prove that each end of $\S$ is asymptotic
(with finite positive multiplicity) to
a totally umbilic annulus, properly embedded in $N$.

\vspace{.15cm}

\vspace{.1cm}

\end{abstract}

\section{Introduction.}

In the celebrated paper~\cite{cm35}, Colding and Minicozzi proved
that complete minimal surfaces of finite topology embedded in $\R^3$
are proper. Based on the proof of this result, Meeks and
Rosenberg~\cite{mr13} showed that complete, connected minimal surfaces
 with positive injectivity radius embedded
in $\R^3$ are proper.
Meeks and Tinaglia~\cite{mt7} then extended both results by proving
that complete surfaces with constant mean curvature embedded in $\mathbb{R}^3$
are proper if they have finite topology or positive injectivity radius.
It is natural to ask to what extent similar
properness results hold for complete surfaces
of finite topology in
other ambient
spaces, where the surfaces are not necessarily embedded or have constant mean curvature.

In this paper we investigate relationships between properness of a complete
immersed surface of finite topology in a 3-manifold of nonpositive sectional curvature
under certain restrictions on
its injectivity radius function and on its mean curvature function.
We first derive area estimates for certain compact
surfaces with bounded absolute mean curvature in Hadamard $3$-manifolds.
We next apply these area estimates to obtain sufficient conditions
for a complete
finite topology surface immersed in a complete $3$-manifold
with nonpositive sectional curvature
to be proper,
as described in Theorem~\ref{thmMain} below. In order to state
this result we need the following definition.

\begin{definition}\label{defAs}{\em
Let $e$ be an end of a complete Riemannian surface $\Sigma$
whose injectivity radius
function we denote by $I_{\Sigma}\colon \Sigma \to (0,\infty)$.
Let
$\cE(e)$ be the collection of proper subdomains $E \subset \S$, with compact
boundary,
that represents $e$.
We define the (lower)
{\em asymptotic injectivity radius of $e$} by
$$I_\Sigma^\infty(e) = \sup\{\inf_{E} I_\Sigma\vert_{E}\mid E \in \cE(e)\}
\in [0,\infty].
$$
}
\end{definition} \vspace{.2cm}

Note that if $\S$ has an end $e$ which admits a one-ended representative $E$,
then $I_\S^\infty(e) = \liminf_EI_\S\vert_E$.
If $\S$ has finite topology, then every end of $\S$ has a one-ended
representative which is an annulus (i.e., a
surface with the topology of $\sn1\times[0,\infty)$), hence this simpler
definition can be used. Moreover,
Lemma~\ref{lemmaLimitInjRad}
in the Appendix shows that if $\S$ has nonpositive
Gaussian curvature and an end
$e$ has a representative $E$ which is an annulus,
then for every divergent sequence of points $\{p_n\}_{n\in\N}$ on $E$,
$$\lim_{n\to \infty}I_\S(p_n)=I_\Sigma^\infty(e).$$

We define the {\em mean curvature function} of
an immersed two-sided surface $\S$ with a given unit
normal field in a Riemannian 3-manifold to be the
pointwise average of its
principal curvatures; note that if $\S$ does not have a unit
normal field, then the absolute value
$\abs{H_\S}$ of the mean curvature function of
$\S$ still makes sense because a unit normal field locally
exists on $\S$ and under a change of this local choice,
the principal curvatures change sign.

\begin{theorem}\label{thmMain}
Let $N$ be a complete $3$-manifold with sectional curvature
$K_N\leq -a^2\leq 0$, for some $a \geq 0$.
Let $\varphi\colon \Sigma \rightarrow N$ be an
isometric immersion of a complete surface $\Sigma$
with finite topology, whose mean curvature function
satisfies $\vert H_\varphi \vert \leq a$. Then $\S$ has nonpositive Gaussian
curvature and the following hold:

\begin{enumerate}[A.]
\item \label{iA} If $N$ is simply connected,
then $I_\S^\infty(e) = \infty$ for every end $e$ of $\S$.
\item \label{iB}
If $N$ has positive injectivity radius $\Inj(N) = \delta > 0$, then every
 end $e$ of $\Sigma$ satisfies $I_\Sigma^\infty(e)\geq \delta$. In particular,
$\Sigma$ has positive injectivity radius.
\item \label{iC}
If $I_\Sigma$ is bounded,
then $\Sigma$ has finite total curvature
$$
\inte{\Sigma}K_\Sigma = 2\pi \chi(\Sigma),
$$
where $\chi(\Sigma)$ is the Euler characteristic of $\Sigma$.
Furthermore, for each annular end representative $E$ of $\Sigma$,
the induced map $\varphi_*\colon \pi_1(E) \rightarrow \pi_1(N)$
on fundamental groups is injective.
\item \label{iD}
If $I_\S^\infty(e) = 0$ for each end $e$ of $\S$, then $\varphi$ is proper.
\end{enumerate}
\end{theorem}

The other main theorem of the paper, Theorem~\ref{thmAsymp} below,
describes, among other things, results on
the asymptotic behavior of complete, properly immersed finite topology surfaces
of constant absolute mean curvature $H\in [0,1)$
in hyperbolic 3-manifolds of finite volume.
Our asymptotic description of these surfaces was inspired
by Theorem~1.1 of Collin, Hauswirth and Rosenberg~\cite{chr2}
who obtained it in the special case that $H=0$.

For any connected, noncompact, orientable surface
of finite topology $S$ different from
an annulus or a plane, there exists a hyperbolic
$3$-manifold $N_S$ of finite volume that admits a properly embedded surface
$\S$, that is totally geodesic in $N_S$, homeomorphic to $S$ and such that each end
of $N_S$ contains at most one end of $\S$; then a ``$t$-parallel''
surface to $\S$ is a properly immersed surface of constant mean curvature
$H(t) = \cos(\arctan(1/t))$. As $t$ ranges from $0$ to $\infty$, this
gives examples with all the possible mean curvatures $H \in [0,1)$. Moreover,
for $t$ sufficiently small, such parallel surfaces can be shown to be embedded,
see~\cite{amr1}.

\begin{theorem}\label{thmAsymp} Let $N$ be a complete, noncompact
hyperbolic $3$-manifold of
finite volume and $H \in [0,\,1)$. Let $\S$ be a
complete, properly immersed surface in $N$
with $\abs{H_\Sigma} \leq H$. Then:

\begin{enumerate}
\item\label{one} $\Sigma$ has finite area and a finite number of connected
components.
\item\label{two} For any divergent sequence of points $\{p_n\}_{n\in\N} \subset \S$,
$\lim_{n\to \infty}I_\S(p_n)=0$.
\item\label{three} $\Sigma$ has total curvature
\begin{equation}\label{totcurv2}
\displaystyle\int_{\Sigma} K_\Sigma = 2\pi\chi (\S).
\end{equation}
\item\label{four}
If $\S$ has infinite topology, then the norm of its second fundamental form
is unbounded.
\item \label{six} If $H_\S=H$, then every annular end representative of an end of $\S$ is
asymptotic (with finite multiplicity) in the $C^2$-norm,
to a totally umbilic annulus properly embedded in $N$. In particular,
if $\S$ has finite topology, then
the norm of the second fundamental form of $\S$ is bounded.
\end{enumerate}
\end{theorem}

The totally umbilic annuli described in the
item~\ref{six} of last theorem are properly
embedded annular ends in $N$
whose lifts to
hyperbolic $3$-space are contained in
equidistant surfaces to totally geodesic planes; see the discussion
in Section~\ref{subAsymp} for a complete description.

An immediate consequence of item~\ref{two} of
Theorem~\ref{thmAsymp} and of item~\ref{iD} of Theorem~\ref{thmMain} is the following
corollary.

\begin{corollary} Let $N$ be a hyperbolic manifold of
finite volume and let $H\in [0,1)$. Then a complete immersed surface
$\Sigma$ in $N$ with finite topology and mean curvature function
$\abs{H_\S} \leq H$ is proper if and only
$I_\S^\infty(e) = 0$ for each end $e$ of $\S$.
\end{corollary}

We remark that the hypothesis $H < 1$ in
the statement of Theorem~\ref{thmAsymp} is a necessary one.
Proposition~\ref{rmksomething} shows that
for any $H\geq 1$, every complete hyperbolic
3-manifold $N$ of finite volume admits a properly immersed, complete annulus $\S$
of constant mean curvature $H$ with positive injectivity radius and infinite area.

The paper is organized as follows.
In Section~\ref{secArea}
we prove isoperimetric inequalities for certain
compact surfaces with boundary in Hadamard $3$-manifolds.
These isoperimetric inequalities yield area estimates which are applied in Section~\ref{secMain}
to prove Theorem~\ref{thmMain}. In Section~\ref{secAsymp} we prove
Theorem~\ref{thmAsymp}.

\section{An isoperimetric inequality in Hadamard manifolds.}\label{secArea}

In this section we obtain an isoperimetric inequality
for certain hypersurfaces in Hadamard manifolds, see Theorem~\ref{thmCyl} below.
By Hadamard manifold we mean a simply connected manifold with nonpositive
sectional curvature.

\begin{theorem}\label{thmCyl} Let $N$ be a Hadamard manifold
of dimension 3 with sectional curvature
$K_N\leq -a^2\leq0$ and let $\Gamma$ be a complete geodesic
of $N$. Given $r>0$, there exists a
$C = C(a,r)>0$ such that every smooth, immersed, compact orientable
surface $\Sigma\subset N$ with mean curvature $\abs{H_\Sigma}\leq a$
and that
stays at a finite distance less than $r$ from $\Gamma$, satisfies
\begin{equation}\label{eqisop}
\A(\Sigma) \leq C\ \len(\partial \Sigma).
\end{equation}
\end{theorem}

\begin{proof}

Let $\Gamma \subset N$ be a complete geodesic, $R = d_N(\cdot,\Gamma)$
be the ambient distance function to $\Gamma$
and $r>0$ be given.
We will prove the theorem using the following claim:

\begin{claim}\label{claimLemma}
Fixed $r>0$,
there exist a smooth function $f\colon [0,r]\to \R$ and constants $C_1,\,C_2 > 0$,
whose construction depends uniquely on the constants $r$ and $a$, such that for all $x \in [0,r]$,
\begin{equation}\label{C1}
0\leq f^\prime(x) \leq C_1
\end{equation}

\noindent holds and that, for every smooth compact surface $\S$ immersed in $R^{-1}([0,r))$
with mean curvature function satisfying $\abs{H_\S}\leq a$, then
\begin{equation}\label{C2}
\Delta_\S(f\circ R) \geq C_2, \quad \text{in}\ \ N.
\end{equation}
\end{claim}

Before we prove the claim, we apply it to obtain the constant $C$
of Theorem~\ref{thmCyl}.
First, note that $R$ is differentiable in $N\setminus \Gamma$,
with $\norma{\grad(R)} = 1$ in $N\setminus \Gamma$. Let $f$ be the function
provided by the claim and let $\S$ be a surface satisfying the hypothesis of the
theorem. Denoting by $\nu$ the conormal vector field along $\partial \S$,
we can apply the divergence theorem
to obtain that
\begin{equation}\label{part1w}
\inte{\S} \Delta_\S(f\circ R) =\inte{\partial \S} (f^\prime\circ R)
\forma{\grad(R),\,\nu} \leq C_1 \, \len(\partial \S),
\end{equation}

\noindent where last inequality comes from \eqref{C1}.
On the other hand, \eqref{C2} implies
\begin{equation}\label{part2w}
\inte{\S}\Delta_\S( f\circ R) \geq C_2 \, \A(\S).
\end{equation}

\noindent By defining $C = C_1/C_2$, \eqref{part1w} and \eqref{part2w}
show that \eqref{eqisop} holds for $\S$, thereby providing the constant $C$
of Theorem~\ref{thmCyl}.

Next, we prove the claim.

\begin{proof}[Proof of Claim~\ref{claimLemma}]
Let $\S$ be as in the claim and let
$f\colon [0,\,r]\rightarrow \R$ be some smooth function, to be chosen a posteriori,
such that $f^\prime\geq0$.
Since $R$ is not smooth in points of $\Gamma$, we
will first show that~\eqref{C2} holds in
$N\setminus \Gamma$; however, our choice of $f$ will be of an even function, then
$f\circ R$ will be smooth in $N$ and~\eqref{C2} will hold everywhere by
continuity.

Consider $\{E_1,\,E_2\}$ an orthogonal frame to $\Sigma$ and let $\eta$
be a normal unitary vector field orienting $\Sigma$.
A straightforward calculation shows that
\begin{eqnarray}
\Delta_\Sigma(f\circ R) \hspace{-3pt}&\hspace{-8pt} = \hspace{-8pt}&\hspace{-3pt}
(f^\prime\circ R) \soma{i=1}{2}\Hess(R) (E_i,\,E_i) \hspace{-2pt}+\hspace{-2pt}
(f^{\prime\prime}\circ R)\soma{i=1}{2}\forma{\grad(R),\,E_i}^2\nonumber\\
&&\vphantom{\displaystyle\sum_{i=1}^ne_i}
+2H_\Sigma (f^\prime\circ R) \forma{\grad(R),\,\eta},\label{eq1}
\end{eqnarray}
\noindent where we denote $\grad = \grad_N$ and
$\Hess(R)(X,\,Y) = \forma{\n_X\grad(R),\,Y}$, respectively,
the gradient and the Hessian with respect to the ambient space.

As $R$ is the distance function to the geodesic $\Gamma$,
the Hessian of $R$ satisfies a matrix valued Riccati type differential equation,
where the independent term is the curvature tensor of $N$.
Then, since $N$ has sectional
curvature satisfying $K_N\leq -a^2\leq 0$,
it follows from the comparison principle to the Riccati equation,
a Hessian comparison principle for the distance function $R$,
given below in \eqref{Hesscomp} (see, for
instance Proposition 5.4 of~\cite{esch1} or the
main result of~\cite{eschHein1}), which we now describe.
For $\rho > 0$, let $C_\rho = R^{-1}(\{\rho\})$ be the
geodesic cylinder of radius $\rho$
around $\Gamma$ and $S_\rho$ be a geodesic sphere of radius $\rho$ centered at a
point $\Gamma(s)$ of $\Gamma$.
Let $\partial_\theta$ be a unitary vector field tangent to
$C_\rho\cap S_\rho$ and let
$\partial_s$ be unitary such that
$\{\grad(R),\,\partial_s,\,\partial_\theta\}$ is an orthonormal frame in $N$,
away from $\Gamma$; see Figure~\ref{figHess}. Then
\begin{figure}
\centering
\includegraphics[width=0.7\textwidth]{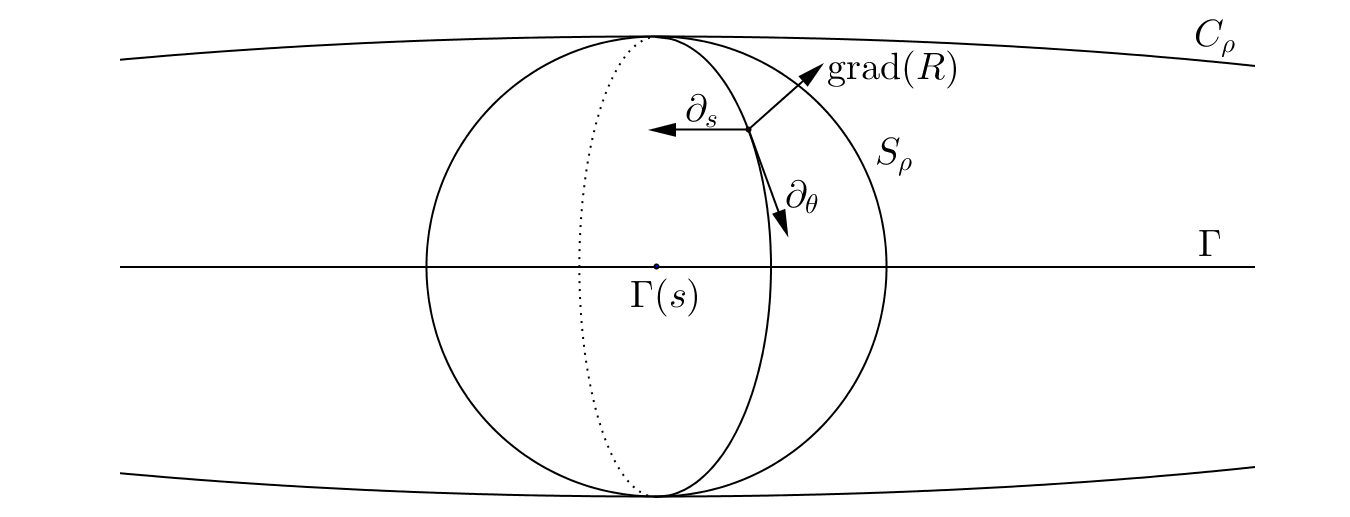}
\caption{The frame $\{\partial_\theta,\,\partial_s,\,\grad(R)\}$.
\label{figHess}}
\end{figure}
\begin{equation}
\Hess (R)(\partial_\theta,\,\partial_\theta) \geq \mu_\theta(R), \quad
\Hess (R)(\partial_s,\,\partial_s) \geq \mu_s(R),\label{Hesscomp}
\end{equation}

\noindent where $\mu_\theta,\,\mu_s$ are the functions that
realize equalities in~\eqref{Hesscomp}
in the spaces of constant sectional curvature $-a^2$, and are defined by

\begin{equation}\label{musmut}
\mu_\theta(x) = \left\{ \begin{array}{l}
a\,\coth(ax),\ \text{if} \ a > 0\\
1/x,\ \text{if}\ a = 0
\end{array}\right., \quad
\mu_s(x) = \left\{ \begin{array}{l}
a\,\tanh(ax),\ \text{if} \ a > 0\\
0,\ \text{if}\ a = 0
\end{array}\right. .
\end{equation}

We use~\eqref{Hesscomp} to estimate $\Hess(R)(E_i,\,E_i)$. Since
$$
\forma{\n_{\partial_s}\grad(R),\,\partial_\theta} = 0 =
\forma{\n_{\partial_\theta}\grad(R),\,\partial_s}
$$
and also
$$
E_i = \forma{E_i,\,\grad(R)}\grad(R) +
\forma{E_i,\,\partial_s}\partial_s+\forma{E_i,\,\partial_\theta}\partial_\theta,
$$

\noindent then~\eqref{Hesscomp} gives
\begin{eqnarray}
\Hess(R) (E_i,\,E_i) \hspace{-3pt}&\hspace{-3pt} =\hspace{-3pt} &\hspace{-3pt}
\forma{E_i,\,\partial_s}^2\Hess(R)(\partial_s,\,\partial_s) +
\forma{E_i,\,\partial_\theta}^2\Hess(R)(\partial_\theta,\,\partial_\theta)\nonumber\\
& \hspace{-3pt}\geq\hspace{-3pt} & \forma{E_i,\,\partial_s}^2\mu_s(R) +
\forma{E_i,\,\partial_\theta}^2\mu_\theta(R).\label{comp1}
\end{eqnarray}

\noindent We sum~\eqref{comp1} for $i=1,\,2$ and use that $\mu_s < \mu_\theta$
to obtain a lower estimate
for the first term of~\eqref{eq1}. We suppress the variable $R$
in the functions $\mu_s$ and $\mu_\theta$
to simplify the notation.
\begin{eqnarray}
\soma{i=1}{2}\Hess(R)(E_i,\,E_i)&\geq&\mu_s \soma{i=1}{2}\forma{E_i,\,\partial_s}^2 +
\mu_\theta \soma{i=1}{2}\forma{E_i,\,\partial_\theta}^2\nonumber\\
& = & \mu_s - \mu_s\forma{\eta,\,\partial_s}^2+
\mu_\theta -\mu_\theta \forma{\eta,\,\partial_\theta}^2\nonumber\\
& \geq & \mu_s + \mu_\theta -\mu_\theta\left(\forma{\eta,\,\partial_s}^2+
\forma{\eta,\,\partial_\theta}^2\right)\nonumber\\
& = & \mu_s + \mu_\theta\forma{\eta,\,\grad(R)}^2. \label{est2}
\end{eqnarray}

Let $\beta$ be the angle between $\grad(R)$ and $\eta$. Since
$f^\prime \geq 0$, then \eqref{est2} and \eqref{eq1} imply that
\begin{eqnarray}
\Delta_\Sigma(f\circ R) & \geq & f^\prime\left(\mu_s + \mu_\theta \cos^2(\beta)\right)+
 f^{\prime\prime}(1-\cos^2(\beta))+2H_\Sigma f^\prime \cos(\beta)\nonumber\\
& = & (f^\prime\mu_\theta - f^{\prime\prime})\cos^2(\beta) + 2H_\Sigma f^\prime\cos(\beta)
+ f^\prime\mu_s +f^{\prime\prime}.\label{eqazero}
\end{eqnarray}

At this point, we are able to finish the proof in the case $a= 0$,
where $\mu_s(R) = 0$, $\mu_\theta(R) = 1/R$ and
$H_\Sigma = 0$. By choosing $f(x) = x^2$, \eqref{eqazero} gives
\begin{equation*}
\Delta_\Sigma(R^2) \geq 2,
\end{equation*}

\noindent and we can set $C_1 = f^\prime(r) = 2r$ and $C_2 = 2$,
which proves the claim and gives the constant $C(0,r)=r$ in the theorem.

Next, assume that $a > 0$. If we suppose that
$f^\prime\mu_\theta - f^{\prime\prime} > 0$
(this will be shown to hold for our choice of $f$), algebraic
manipulation in \eqref{eqazero} gives
\begin{eqnarray}
\Delta_\Sigma(f\circ R) & \geq & (f^\prime\mu_\theta -
f^{\prime\prime})\left(\cos(\beta) + H_\Sigma\frac{f^\prime}{f^\prime\mu_\theta -
f^{\prime\prime}}\right)^2\nonumber\\
&&-
H_\Sigma^2\frac{(f^\prime)^2}{f^\prime\mu_\theta - f^{\prime\prime}}
+ f^\prime\mu_s + f^{\prime\prime}\nonumber\\
& \geq & f^\prime\mu_s + f^{\prime\prime}
- H_\Sigma^2\frac{(f^\prime)^2}{f^\prime\mu_\theta - f^{\prime\prime}}.\label{eqsemifinal}
\end{eqnarray}

Using that $\mu_s \mu_\theta = a^2$ and $H_\Sigma^2 \leq a^2$,
we obtain the following estimate

\begin{eqnarray}
\Delta_\Sigma(f\circ R) & \geq &\frac{(f^\prime)^2(a^2-H_\Sigma^2) +
 f^{\prime\prime}(f^\prime\mu_\theta - f^{\prime\prime}
 - f^\prime\mu_s)}{f^\prime\mu_\theta - f^{\prime\prime}}\nonumber\\
& \geq &f^{\prime\prime}\left(1 - \frac{f^\prime\mu_s}{f^\prime\mu_\theta
- f^{\prime\prime}}\right).\label{ineqmain}
\end{eqnarray}

We now make our choice of $f$. Fix $k \in \N$
and let $f_k\colon \R \rightarrow \R$ be a solution to
\begin{equation}\label{edo}
f^\prime(x)\mu_s(x) = \frac{k}{k+1}(f^\prime(x)\mu_\theta(x) - f^{\prime\prime}(x)),
\end{equation}
\noindent which can be explicitly expressed by
\begin{equation}\label{deff}
f_k(x) = - \frac{k}{a(\cosh(ax))^{\frac{1}{k}}}.
\end{equation}

\noindent The derivatives of $f_k$ are

\begin{eqnarray}
f_k^\prime(x)& = &\frac{\tanh(ax)}{(\cosh(ax))^\frac{1}{k}}, \label{fprime}\\
f_k^{\prime\prime}(x)& = &
\frac{a}{k(\cosh(ax))^{\frac{2k+1}{k}}}[k+1-\cosh^2(ax)],\label{fdprime}
\end{eqnarray}

\noindent and, by~\eqref{edo} and~\eqref{fprime}
we conclude that $f_k^\prime\mu_\theta- f_k^{\prime\prime} > 0$;
hence,~\eqref{ineqmain} gives the inequality
\begin{equation}\label{defg}
\Delta_\S(f\circ R) \geq a\frac{k+1-\cosh^2(aR)}{k(k+1)(\cosh(aR))^{\frac{2k+1}{k}}}.
\end{equation}

Let $g_k\colon [0,r]\to \R$ be defined by the expression
in the right hand side of \eqref{defg}, i.e.,
\begin{equation*}
g_k(x) = a\frac{k+1-\cosh^2(ax)}{k(k+1)(\cosh(ax))^{\frac{2k+1}{k}}}.
\end{equation*}
To prove the claim
it suffices to find a lower positive
bound for $g_k$, as $f_k^\prime$ is nonnegative and bounded
by $C_1 = 1$. Note that $g_k(x) > 0$ if and only if
\begin{equation*}
\cosh^2(ax) < k+1.
\end{equation*}

\noindent Choose $n \in \N$ such that $\cosh^2(ar) < n+1$. Then $f_n$
is a function that satisfies \eqref{C1} in $[0,\,r]$, for $C_1 = 1$.
A direct calculation of the derivative of $g_n$ shows
that it is a decreasing function in $[0,r]$; thus,
if we set $C_2 = g_n(r) > 0$,
it follows that $g_n(x) \geq C_2 > 0$ for $x \in [0,r]$.
Since $n$ was chosen independently of $\Sigma$ and
$C_1,\,C_2$ depend uniquely on $f_n$, this completes the proof
of the claim in the case when $a > 0$.\end{proof}

\noindent As explained immediately after the statement of
Claim~\ref{claimLemma}, the theorem
is now proved. \end{proof}

\begin{remark}{\rm
The hypothesis $\abs{H_\Sigma}\leq a$ in Theorem~\ref{thmCyl}
cannot be improved. Indeed, if $N$ is simply connected, of constant
sectional curvature $K_N = -a^2\leq 0$, then, for each $H > a$, there exists a
geodesic cylinder
of constant mean curvature $H$, containing compact subdomains
with constant boundary length and with arbitrarily large area.
}\end{remark}

The next corollary is an immediate consequence of
Theorem~\ref{thmCyl}. One only needs to check that if
the boundary of the compact surface $\S$ given below has length
$L$, then $\S$ is contained in a solid geodesic cylinder of radius
$L/2$, which follows from the mean curvature comparison principle.

\begin{corollary}[Area estimate for surfaces
with two boundary components]\label{corArEst}
Let $N$ be a Hadamard $3$-manifold with sectional
curvature $K_N\leq -a^2\leq0$. Then, for
each $L > 0$, there is a constant $C = C(a,L)$
such that the following holds.

Let $\Sigma \subset N$ be a compact surface
immersed in $N$ with $\abs{H_\Sigma}\leq a$. If the boundary of
$\S$ consists on one or two components
and has total length at most $L$, then
\begin{equation}\label{areaest}
\A(\Sigma) \leq C\, L.
\end{equation}

\end{corollary}

\section{The proof of Theorem~\ref{thmMain}.}\label{secMain}
Throughout this section $N$ will be a complete Riemannian $3$-manifold
with sectional curvature $K_N\leq -a^2\leq 0$,
and $\varphi\colon \Sigma\rightarrow N$ will be an
isometric immersion of a finite topology
surface $\S$ such that
the mean curvature function $H_\varphi$ of the immersion satisfies
$\abs{H_\varphi}\leq a$.
A simple consequence of the Gauss equation is that
the Gaussian curvature function $K_\S$ of $\Sigma$ is nonpositive;
hence,
if $p \in \S$ and $I_\S(p)$ is finite,
then there is a closed geodesic loop based at $p$ in $\S$ of length $2I_\S(p)$
(see, for instance, Proposition~2.12, Chapter~13 of~\cite{doc2}).
Moreover, it follows from the Gauss-Bonnet formula that such loop is homotopically
nontrivial in $\S$. The existence of such loops will be used in the proofs of the next two
propositions.

\begin{proposition}\label{propInjPos} Suppose $N$ is simply connected and
$E \equiv \sn1\times [0,+\infty)$ is a complete,
noncompact Riemannian annulus.
If $\varphi\colon E\rightarrow N$ is an isometric immersion with
$\abs{H_\varphi}\leq a$, then the asymptotic injectivity radius of $E$,
which we denote by $I_E^\infty$, is infinite.
In particular, item~\ref{iA} of Theorem~\ref{thmMain} holds.
\end{proposition}

\begin{proof}
An elementary calculation shows that
there is an $\ve>0$ independent of $E$ such that intrinsic
balls $B_E(p,\ve)\subset E-\partial E$
have area greater than some fixed positive constant; see
Theorem~3 and Remark~4 in the appendix of~\cite{frn1}.
In particular, since $E$ is noncompact, complete and has compact boundary, then $E$
has infinite area.

Arguing by contradiction, suppose that
the asymptotic injectivity radius of $E$ is
$I_E^\infty = L \in [0,\, \infty)$. By the definition
of $I_E^\infty$,
there is an intrinsically divergent sequence of points $\{q_n\}_n$ in $ E$
such that $$\lim_{n\to \infty}I_E(q_n)=L.$$
After replacing by a subsequence,
$I_E(q_n)< L+1$,
$\{d_E(q_n,\,\partial E)\}_n$ is increasing,
$d_E(q_1,\partial E) \geq L +1$ and
$d_E(q_n,\,q_{n+1}) \geq 2 L+2$.
Hence, there exist homotopically
nontrivial geodesic loops $\g_n$ with
base points $q_n$ and lengths equal to $2I_E(q_n)< 2L+2$,
which are pairwise disjoint by the triangle inequality.

For $n > 1$, let $E_n$ be the compact annular region of $E$ bounded by $\gamma_{1}$
and $\gamma_n$. Since the total length of $\partial E_n$ is less
than $4L+4$, it follows from Corollary~\ref{corArEst}
that there is a constant $C$ such that $E_n$ satisfies
the uniform area estimate~\eqref{areaest}:
\begin{equation*}
\A(E_n) \leq C(4L+4).
\end{equation*}
\noindent In particular, $E$ has finite area,
which contradicts our previous observation
that the area of $E$ was infinite, and completes the proof that $I_E^\infty = \infty$.
\end{proof}

\begin{proposition}\label{propImportant} Suppose
$\varphi\colon E\rightarrow N$ is an isometric
immersion of a complete annulus $E\equiv \sn1\times[0,\,+\infty)$ in $N$
 satisfying $\abs{H_\varphi} \leq a$. Then:

\begin{enumerate}[I.]
\item \label{i1}
If $I_E^\infty \in [0,\infty)$, then the induced homomorphism
$\varphi_*\colon \pi_1(E)\rightarrow \pi_1(N)$ is injective.
\item \label{i2}
 $I^\infty_E \geq \Inj(N)$, where $\Inj(N)$ denotes the injectivity radius of $N$.
\item \label{i3}
If $I^\infty_E = 0$, then $\varphi$ is proper.
\end{enumerate}
\end{proposition}

\begin{proof}
We first prove item~\ref{i1} of the proposition.
Arguing by contradiction, suppose
$I_E^\infty$ is finite
and $$\varphi_*\colon \pi_1(E)\rightarrow \pi_1(N)$$ is not injective.
Since $\pi_1(E)$ is isomorphic to $\Z$, the kernel $K$ of $\varphi_* $ is a cyclic
subgroup of index $n$, for some $n\in \N$. Let
$\Pi\colon \wt{E}\to E$ be the $n$-sheeted covering space of
$E$ corresponding to the subgroup $K$. Note that $\wt{E}$ is
 an annulus of nonpositive Gaussian curvature,
 $I_{\wt{E}}^\infty$ is less than or equal to $nI_E^\infty$
(since $y\in \Pi^{-1}(x)\Rightarrow I_{\wt{E}} (y)\leq n I_E(x)$)
and the induced map from the fundamental group of $\wt{E}$ to $N$ is trivial.
By covering space theory, $(\varphi\circ \Pi)\colon \wt{E}\to N$ lifts isometrically
to the universal cover $\wt{N}$ of $N$, which is a Hadamard manifold with respect to the
pulled-back metric. Since $I_{\wt{E}}^\infty$ is finite,
Proposition~\ref{propInjPos} gives a contradiction,
thereby completing the proof of item~\ref{i1}.

We next prove statement {\ref{i2}}. Suppose that $I^\infty_E < \Inj(N) $. Then there
exists a geodesic loop $\g$ in $E$ with length less than $2\Inj(N)$.
Since $\g$ is homotopically
nontrivial in $E$ and lies in a simply connected ball in $N$, then for any $p\in \g$,
the induced map
$\varphi_*\colon \pi_1(E,p)\rightarrow \pi_1(N,p)$ is trivial, contradicting item~\ref{i1}.

Finally, we prove {\ref{i3}}.
If $\varphi$ were not proper, there would exist an intrinsically
divergent sequence of points $q_n \in E$,
such that $\varphi(q_n)$ converges to a point $q \in N$.
Lemma~\ref{lemmaLimitInjRad} in the Appendix implies
$\lim_{n\to \infty} I_E(q_n)= 0$; hence, after
replacing by a subsequence, $q_n\in B_N(q,I_N(q)/2)$
and there exist homotopically nontrivial geodesic loops
$\gamma_n$ based at $q_n$ with lengths $2I_E(q_n)<I_N(q)$. By the triangle inequality,
the loops $\varphi(\gamma_n)$
are contained in the simply connected geodesic ball $B_N(q,I_N(q))$,
contradicting item~{\ref{i1}}. Thus, $\varphi$ is proper.\end{proof}

All of the assertions in Theorem~\ref{thmMain}, except for the
first statement of item~\ref{iC} of Theorem~\ref{thmMain},
follow immediately from Propositions~\ref{propInjPos} and~\ref{propImportant}.
The Cohn-Vossen~\cite{cv1} inequality implies that for
any complete surface $M$ of nonpositive curvature,
\begin{equation} \label{CV}
\int_{M} K_M \leq 2\pi \chi(M).
\end{equation}
In our setting where $M=\S$ and each end $e$ of
$\S$ satisfies $I_\S^\infty(e)$ is bounded,
Theorem~11 of Huber~\cite{hu1} implies
\begin{equation*}
\int_{\S} K_\S = 2\pi \chi(\S),
\end{equation*}
which completes the proof of Theorem~\ref{thmMain}.

\section{Proof of Theorem~\ref{thmAsymp}.}\label{secAsymp}

In this section, we prove Theorem~\ref{thmAsymp}, which describes,
among other things, the asymptotic behavior
of certain immersed surfaces in hyperbolic manifolds of finite volume.
Before we prove this result, we set up the notation
that we use and give a brief review of the structure of the ends of
orientable hyperbolic
manifolds of finite volume, called {\em cusp ends}.

We will use the half-space model for the hyperbolic space:
$$\hn3 = \{(x,\,y,\,z)\in \R^3\mid z > 0\},$$

\noindent endowed with the metric
$\displaystyle ds^2 = \frac{dx^2+dy^2+dz^2}{z^2}$.
In this model, the horizontal planes
\begin{equation*}
\cH(t) = \{(x,y,z)\in \hn3\mid z = t\}
\end{equation*}

\noindent are horospheres, with constant mean curvature $1$
with respect to its upward pointing unit normal field.
Vertical planes are totally geodesic and isometric to the
hyperbolic plane $\hn2$.

Fix two linearly independent horizontal vectors $u = (u_x,\,u_y,\,0),\,v =
(v_x,\,v_y,0) \in \R^3$
and let $G(u,v)$ be the group of parabolic translations of $\hn3$ generated
by $u,\,v$, i.e., $G(u,\,v) = \{\tau(m,n)\mid (m,n) \in \Z\times\Z\}$, where
each $\tau(m,n)$ is the parabolic isometry of $\hn3$ defined by
$\tau(m,n)(p) = p+mu+nv$. In coordinates,
\begin{equation}\label{taugeneral}
\begin{array}{rccl}
\tau(m,n)\colon& \hn3& \to & \hn3\\
&(x,y,z)&\mapsto & (x+mu_x+nv_x,y+mu_y+nv_y,z).
\end{array}
\end{equation}

If $N$ is a complete, orientable, noncompact hyperbolic 3-manifold of finite
volume, it has a finite number of ends $\cC_i$,
$i = 1,\,2,\,\ldots, n$, called the \emph{cusp ends} of $N$.
For each $\cC_i$ there exists $t_i > 0$ and
linearly independent horizontal vectors $u_i,\,v_i$,
such that $\cC_i$ is represented by, and henceforth
isometrically identified with, the quotient of
\begin{equation}\label{Mc}
\cM(t_i) = \bigcup_{t \geq t_i} \cH(t) = \{(x,y,z)\in \hn3 \mid z \geq t_i\}
\end{equation}
\noindent by the action of the group $G(u_i,\,v_i)$.
Since $G(u_i,v_i)$ leaves every horosphere $\cH(t)$ invariant
and $\cM(t_i)$ is foliated by $\{\cH(t)\}_{t\geq t_i}$,
each $\cC_i$ admits a product foliation by the family of
constant mean curvature 1 tori $\{\cT(t) = \cH(t)/G(u_i,v_i)\}_{t \geq t_i}$.
Also, the fundamental group
of each $\cC_i$ is naturally isomorphic to $G(u_i,v_i)$, viewed as the
subgroup of isometries of $\cM(t_i)$ that
commute with the covering map
$\psi_i\colon\cM(t_i)\to\cC_i$.

\subsection{Proof of items~\ref{one} -- \ref{four} of Theorem~\ref{thmAsymp}.}

With the notation concerning the structure of the cusp ends of $N$
discussed above, we next prove
the first four items of Theorem~\ref{thmAsymp}.

Let $\varphi\colon\Sigma\rightarrow N$ be the immersed surface given in the statement
of the theorem. Let $\cC_1,\,\cC_2,\,\ldots,\cC_n$ be cusp end representatives
of the ends of $N$ and let
$N_T = \ol{N \setminus\left(\cC_1\cup \cC_2\cup \ldots\cup \cC_n\right)}$.
Since $\varphi$ is a proper map and $N_T$ is compact, it holds that
$\varphi(\S) \cap N_T$ is compact. Without loss of generality, we may assume
that $\partial N_T$ is transverse to $\varphi(\S)$ and consists of a finite
collection of immersed closed curves.

To prove that $\S$ has finite area, it suffices to show that each intersection
$\varphi(\S) \cap \cC_i$ has finite area.
Let $\cC$ be one of the cusp ends $\cC_i$.
Up to a reparameterization, we may assume that $\cC = \cup_{t\geq 1} \cT(t)$.
We define $E(\cC) = \varphi^{-1}(\cC) \subset \S$ and
$E(t) = \varphi^{-1}\left(\cup_{s\in[1,t]}\cT(s)\right)$. We also use the
notation $\varphi = \varphi\vert_{E(\cC)}$.

Let $R\colon \cC \to [0,\,\infty)$
be the distance function to $\cT(1)$.
By following the arguments in the proof of Theorem~\ref{thmCyl},
one can obtain under the assumptions of the theorem
that the intrinsic Laplacian of $R\circ \varphi$ satisfies
\begin{equation}\label{superharmonic}
\Delta_{E(\cC)} (R \circ \varphi) \leq H^2-1 < 0.
\end{equation}

Let $t > 1$ be a regular value for $\cR\circ \varphi$
and let $\Gamma_t = \varphi^{-1}(\cT(t))$. Also, we denote
$\Gamma_1 = \partial E(\cC)$, then
$\partial E(t) = \Gamma_1 \cup \Gamma_t$.
Integrating \eqref{superharmonic} over $E(t)$, we obtain
\begin{equation}\label{part1}
\inte{E(t)}\Delta_{E(\cC)}(R \circ \varphi) \leq (H^2-1)\A(E(t)).
\end{equation}

\noindent Applying the divergence theorem to the left hand side of \eqref{part1},
gives
\begin{equation}\label{part2}
\inte{E(t)}\hspace{-8pt}\Delta_{E(\cC)}(R \circ \varphi) = \inte{\Gamma_1}\hspace{-4pt}
\forma{\grad_{E(\cC)}(R \circ \varphi),\,\nu_1}
+ \inte{\Gamma_t}\hspace{-4pt}\forma{\grad_{E(\cC)}(R \circ \varphi),\,\nu_t},
\end{equation}

\noindent where $\nu_1$ and $\nu_t$ denote respectively
the outward pointing conormal vectors to $E(t)$ along $\Gamma_1$ and
$\Gamma_t$ (see Figure~\ref{figEt}).
It follows from \eqref{part1} and \eqref{part2} that
\begin{figure}
\centering
\includegraphics[width=0.6\textwidth]{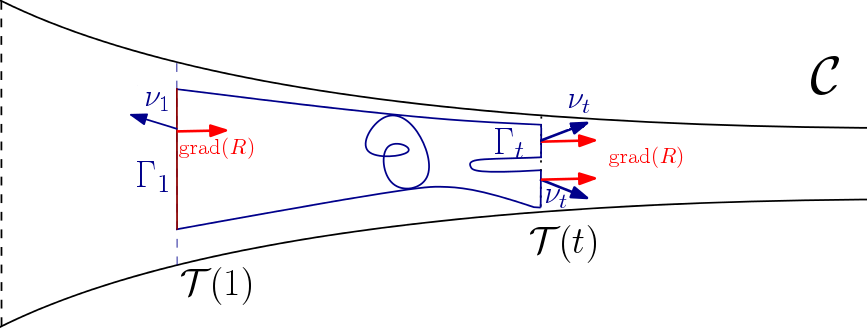}
\caption{The immersed compact region $\varphi(E(t))$.}\label{figEt}
\end{figure}
\begin{equation}\label{estim1}
(1-H^2)\A(E(t))\leq - \inte{\Gamma_1}\hspace{-1pt} \forma{\grad_{E(\cC)}(R \circ \varphi),\,\nu_1}
- \inte{\Gamma_t}\hspace{-1pt}\forma{\grad_{{E(\cC)}}(R \circ \varphi),\,\nu_t}.
\end{equation}

Since $\varphi$ is an isometric immersion and $\nu_t$ is tangent to $E(\cC)$,
$\forma{\grad_{E(\cC)}(R\circ \varphi),\,\nu_t} =
\forma{\grad(R),\,d\varphi(\nu_t)}$. Moreover, by the definition of $E(t)$,
along $\varphi(\Gamma_t)$ we have
$$\forma{\grad(R),\,d\varphi(\nu_t)}>0.$$

\noindent Hence, $\forma{\grad_{E(\cC)}(R\circ \varphi),\,\nu_t} > 0$
and \eqref{estim1} implies that
\begin{equation}\label{areaest2}
\A(E(t)) < -\frac{1}{1-H^2} \inte{\Gamma_1}
\forma{\grad_{E(\cC)}(R \circ \varphi),\,\nu_1}
\leq \frac{1}{1-H^2}\len(\Gamma_1).
\end{equation}

\noindent It follows that $\A(E(\cC))$ is bounded by $\frac{1}{1-H^2}\len(\Gamma_1)$,
which proves that $\S$ has finite area.

To finish the proof of item~\ref{one}, just note that each connected component
$E$ of $\S$ must be such that $\varphi(E) \cap N_T\neq \es$, otherwise
$\varphi(E)$ would be contained in a cusp end $\cC$ of $N$; since
$\varphi$ is proper, this implies that $R \circ \varphi\vert_E$ would attain a
minimal value $t^*$ on an interior point $p^*\in E$. Then, the
mean curvature comparison principle applied to
$\cT(t^*)$ and to $\varphi(E)$ at $\varphi(p^*)$ gives a contradiction, since
the mean curvature of $\cT(t^*)$ is $1$, $\varphi(E)$ lies in the mean
convex side of $\cT(t^*)$ and the mean curvature of $\varphi(E)$ at
$\varphi(p^*)$ is strictly less than $1$. Finally, since $\varphi$ is proper,
$\varphi^{-1}(N_T)$ must contain a finite number of connected components.

The second item of the theorem follows from item~\ref{one}, as we next explain.
Suppose there exist an $\ve>0$ and
a divergent sequence of points $\{p_n\}_{n\in\N}$
in $\S$ such that $I_\S(p_n) \geq \ve$. After replacing by a subsequence, we may
assume that $\{B_\S(p_n,\,\ve)\}_{n\in\N}$ is a collection of pairwise disjoint
disks. Since $\S$ has nonpositive curvature,
comparison theorems imply that $\A(B_\S(p_n,\,\ve))\geq \pi\ve^2$.
Hence $\S$ has infinite area, contradicting item~\ref{one}, which proves
item~\ref{two}.

Next, we prove item~\ref{three}. Since $\S$ has a finite number of connected
components by item~\ref{one},
if $\S$ has infinite topology, then its Euler characteristic is $\chi(\S) = -\infty$.
Hence, Cohn-Vossen inequality \eqref{CV}
implies that $\S$ has infinite total curvature, which proves item~\ref{three}
in this case. In the case where $\S$ has finite topology, item~\ref{two}
implies that $I_\Sigma$ is bounded; therefore equation
\eqref{totcurv2} holds by item~\ref{iC} of Theorem~\ref{thmMain}.

To prove that $\varphi(\S)$ has unbounded norm of the second fundamental form
when it has infinite topology, just note that a uniform bound on
$\norma{A_\varphi}$, with the assumption that $\abs{H_\varphi}\leq H$,
would imply that $K_\S$ is uniformly bounded, by Gauss' equation. In
particular, since $\S$ has finite area it would have finite total
curvature, contradicting item~\ref{three}, which proves item~\ref{four}.

\subsection{Bounds on the second fundamental form.}

A key step in obtaining the asymptotic
description of the annular ends of $\S$ in the constant mean curvature setting,
is the boundedness of the second fundamental form of each such end.

\begin{proposition}\label{propbound}
Let $N$ be a complete, noncompact hyperbolic $3$-manifold of
finite volume and $H \in [0,\,1)$. If $E$ is a
complete, properly immersed annulus in $N$
with constant mean curvature $H$, then
$\varphi(E)$ has bounded norm
of its second fundamental form.
\end{proposition}

\begin{proof}
Since $\varphi$ is a proper map, then there exist a cusp end $\cC$ of $N$ and
a subannular end $E^\prime \subset E$ such that
$\varphi(E^\prime)$ is contained in $\cC$; since it suffices to prove
that $E^\prime$ has bounded second fundamental form, we may assume
that $\varphi(E) \subset \cC$.
Up to a reparameterization, we write $\cC = \cup_{t\geq 1} \cT(t)$; in
particular, $\varphi(\partial E) \subset \cup_{t\in[1,\,t_0]}\cT(t)$,
for some $t_0 \geq 1$. Using this notation, we
prove next claim.

\begin{claim}\label{claimtransverse}
For almost every $t \geq t_0$, $\varphi(E)$ meets the torus $\cT(t)$
transversely and for each such $t$ there is a unique homotopically nontrivial
closed curve $\a_t \subset E$ such that $\varphi(\a_t)$ is contained in $\cT(t)$.
Moreover, the annular subend $E^\prime$ of $E$ determined by $\a_t$
is immersed in $\cup_{s \geq t}\cT(s)$ and the induced map
$\varphi'_*\colon \pi_1(E')\to \pi_1(\cup_{s \geq t}\cT(s))$ is injective.
\end{claim}

\begin{proof}[Proof of Claim~\ref{claimtransverse}]
Sard's Theorem implies that for almost all $t\geq t_0$, $\varphi$
is transverse to $\cT(t)$ and for such $t$,
$\G_t=\varphi^{-1}(\cT(t))$ is a finite collection of
pairwise disjoint simple closed curves in $E$.
Fix a regular value $t_1>t_0$.
We next prove that exactly one of the curves in $\G_{t_1}$ is homotopically
nontrivial in $E$.
Since $\varphi^{-1}(\cup_{t\in[1,t_1]}\cT(t))$ is a
compact surface (possibly disconnected)
containing $\partial E$,
then the 1-cycle $\partial E$ is $\Z_2$-homologous to
the collection of unoriented curves in $\G_{t_1}$.
Since $\partial E$ represents the nontrivial element
in $H_1(E,\Z_2)$, then at least one of the curves in $\G_{t_1}$
is homotopically nontrivial.

Concerning the uniqueness part of the claim,
assume that there are two homotopically nontrivial curves,
$\a_1,\a_2$, in $E$ such that $\varphi(\a_1)$ and $\varphi(\a_2)$ are both contained
in $\cT(t_1)$.
Consider the subends $E_1,E_2 \subset E$ determined respectively by
$\a_1,\,\a_2$ and assume
 that $E_2 \subset E_1$. Let $\widehat{E} \subset E$
be the compact annulus bounded by $\a_1$ and $\a_2$. Then,
either $\varphi(\widehat{E})$
intersects $\cup_{t< t_1}\cT(t)$ or it is contained in
$\cup_{t\geq t_1}\cT(t)$. If the first case occurs, there is a $t_* < t_1$
such that $\varphi(\wh{E})$ intersects $\cT(t_*)$ but does not
intersect $\cT(t)$ for any $t<t_*$. Because
the mean curvature vector of every torus $\cT(t)$
has length $1$ and points into the cusp subend of $\cC$ determined
by $\cT(t)$, the
mean curvature comparison principle for the
surfaces $\varphi(\wh{E})$ and $\cT(t_*)$ at a point
in $\varphi(\wh{E})\cap \cT(t_*)$
gives a contradiction, therefore,
$\varphi(\widehat{E}) \subset \cup_{t\geq t_1}\cT(t)$.
Since $\varphi(E)$ meets $\cT(t_1)$ transversely,
it follows that $\varphi(E_2)$ contains
points in $\cup_{t< t_1}\cT(t)$ near $\varphi(\a_2)$, and
we may find $t^\prime_* <t_1$, where $t^\prime_*$ is
the smallest $t$ such that $\varphi(E_2)$ intersects $\cT(t)$,
which by the previous argument gives a contradiction.
This proves that $\varphi$ induces an immersion
$\varphi'\colon E'\to \cup_{s \geq t}\cT(s)$.

By item~\ref{two} of Theorem~\ref{thmAsymp}, $I_\S$ is bounded.
Thus, by part C of Theorem~\ref{thmMain},
$\varphi'_*\colon \pi_1(E')\to \pi_1(\cup_{s \geq t}\cT(s))$ is injective.
This completes the proof of Claim~\ref{claimtransverse}.
\end{proof}

\begin{figure}
\centering
\includegraphics[width=0.9\textwidth]{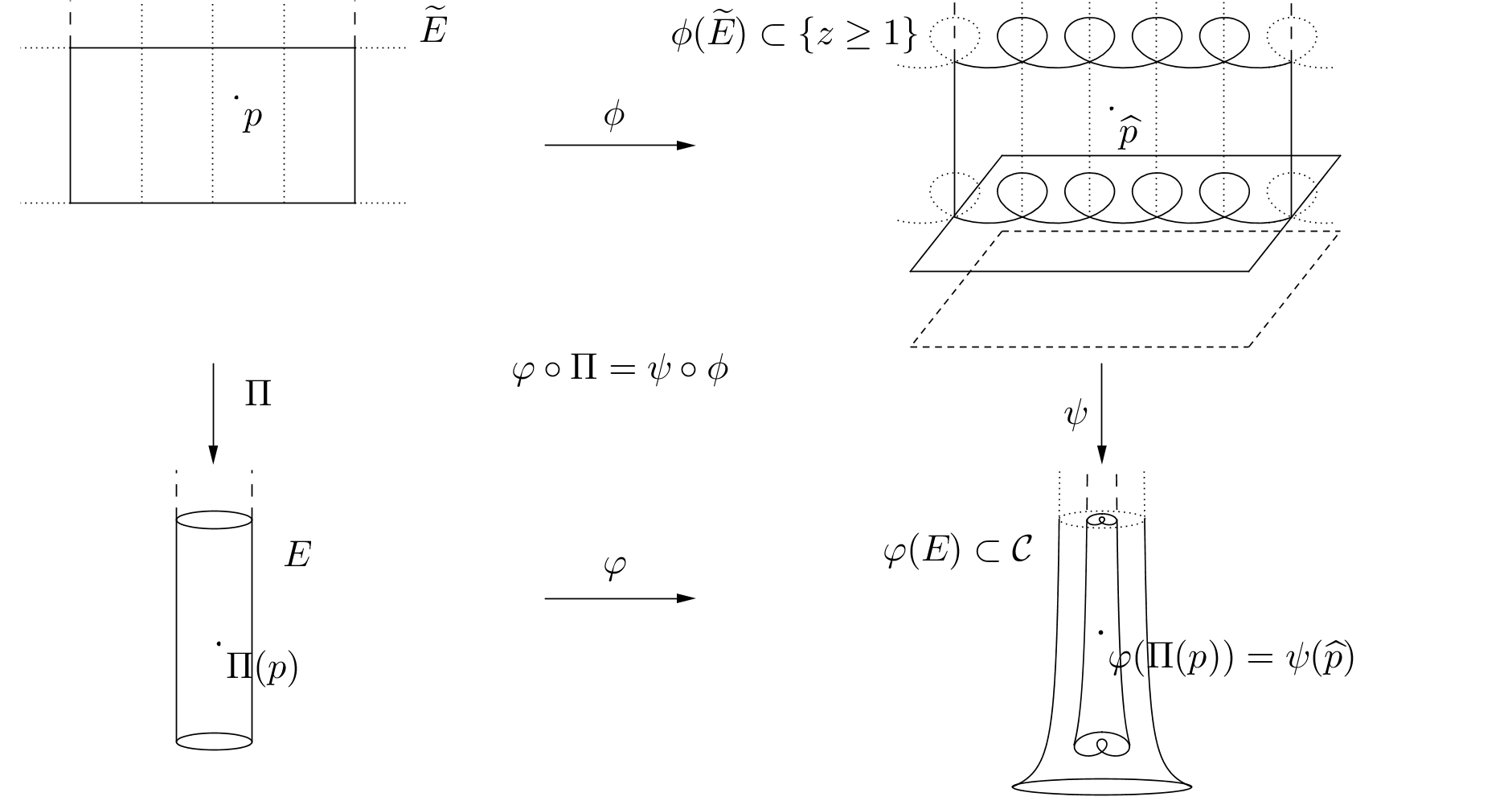}
\caption{The immersions $\phi$ and $\varphi$ and the covering maps
$\Pi\colon \Et \to E$ and $\psi\colon \cM(1)\to \cC$.}\label{figDiagram}
\end{figure}

We next fix some notation. Because of Claim~\ref{claimtransverse}
we henceforth assume, without loss of generality, that $t_0 =1$ and that
$\varphi(E)$ intersects $\partial \cC=\cT(1)$
transversely in the set $\partial E = \alpha_1$. Moreover, we
assume that $\cC$ is isometric to the quotient space $\cM(1)/G(u,v)$, for
two linearly independent horizontal vectors $u$ and $v$. Let
$\psi\colon \cM(1) \to \cC$
be the covering of $\cC$ associated to $G(u,v)$ and let
$\Pi\colon\Et \to E$ be the universal cover of $E$.
Choose a base point $p \in \Et$ and consider $\varphi(\Pi(p)) \in \cC$.
After choosing
$\wh{p} \in \psi^{-1}(\varphi(\Pi(p)))$,
covering space theory implies that
there exists a unique
immersion $\phi\colon \Et \to \cM(1)$ such that
$\phi(p) = \wh{p}$ and
$\psi \circ\phi = \varphi \circ \Pi$; in particular, it follows that $\phi$ is
proper, since $\varphi_*\colon \pi_1(E)\to \pi_1(\cup_{s \geq t}\cT(s))$
is injective.
Moreover, $\phi(\Et)$ is an immersed half-plane in $\hn3$.
See Figure~\ref{figDiagram}.

To prove Proposition~\ref{propbound},
it suffices to prove that $\phi(\Et)$ has bounded norm
of the second fundamental form $\norma{A_{\phi}}$.
Arguing by contradiction, assume there exists a sequence of points
$\{p_n\}_{n\in\N}$ in $\Et$ such that $\norma{A_\phi}(\wh{p}_n) \geq n$,
where we denote $\wh{p}_n = \phi(p_n)$.
Since $\norma{A_\varphi}$ is bounded on compact sets, the sequence of the
image points $\{\psi(\wh{p}_n)\}_{n\in\N}\subset \varphi(E)$ is intrinsically
(thus extrinsically, since $\varphi$ is proper) divergent.
After choosing a subsequence, we may assume
$d_{\cC}(\psi(\wh{p}_n)),\,\psi(\wh{p}_m)) > 2$,
for $n\neq m$; hence the sequence $\{\wh{p}_n\}_{n\in\N}$
in $\hn3$ is extrinsically divergent and $d_{\hn3}(\wh{p}_n,\,\wh{p}_m) > 2$
for $n\neq m$. For the construction that follows, see Figure~\ref{figCnGn}.

Without loss of generality, we may assume that the spheres $\{\partial\ol{B}_{\hn3}(\wh{p}_n,1)\}_n$
are transverse to $\phi$.
For $n\in\N$, let $C_n \subset \wt{E}$ be the connected component of
$\phi^{-1}(\ol{B}_{\hn3}(\wh{p}_n,1))$ containing $p_n$. Consider the function
$$\begin{array}{rccl}
f_n\colon& C_n&\to &\R\\
&x&\mapsto& \norma{A_\phi}(\phi(x))d_{\hn3}(\phi(x),\,\partial \overline{B}_{\hn3}(\wh{p}_n,\,1)).
\end{array}$$

\noindent Let $q_n \in C_n$
be a point where
$f_n$ achieves its maximum. Let $\wh{q}_n = \phi(q_n)$. Then,
\begin{equation}\label{almaxcurv}
\norma{A_\phi}(\wh{q}_n)d(\wh{q}_n,\,\partial B_{\hn3}(\wh{p}_n,\,1))
= f_n ({q}_n) \geq f_n({p}_n) = \norma{A_\phi}(\wh{p}_n) \geq n.
\end{equation}

\noindent Let $\lambda_n = \norma{A_\phi}(\wh{q}_n)$ and
$\delta_n = d_{\hn3}(\wh{q}_n,\,\partial B_{\hn3}(\wh{p}_n,\,1))$.
Note that, if $x\in C_n$ and $ \phi(x) \in \ol{B}_{\hn3}(\wh{q}_n,\,\delta_n/2)$,
then $\norma{A_\phi}(\phi(x)) \leq 2\lambda_n$,
since $d_{\hn3}(\phi(x),\wh{q}_n) < \delta_n/2$.

In this proof we use the following notation: for $\lambda > 0$
and a Riemannian manifold $M = (M,g)$, we denote $\lambda M = (M,\,\lambda^2 g)$ the
Riemannian manifold given by a scaling of the metric of $M$ by $\lambda$.

Using exponential coordinates in $\hn3$
centered at the point $\wh{q}_n$, consider $\lambda_n B_{\hn3}(\wh{q}_n,\,\delta_n/2)$
to be a ball of radius
$\displaystyle {\lambda_n \delta_n}/{2}$ in $\R^3$ with $\wh{q}_n$ at the origin.
From \eqref{almaxcurv}, we have
$\lambda_n\delta_n \rightarrow \infty$,
and so the sequence $\{\lambda_nB_{\hn3}(\wh{q}_n,\,\delta_n/2)\}_{n\in\N}$ of Riemannian balls
converges to the Euclidean space $\R^3$ with its flat metric.
Let $\wh{C}_n = \phi(C_n)\cap \ol{B}_{\hn3}(\wh{q}_n,\delta_n/2)$
and, for $r> 0$, let
$B_n(r) =B_{\lambda_n\hn3}(\wh{q}_n,r) = \lambda_nB_{\hn3}(\wh{q}_n,\,r/\lambda_n)$.
The scaled surfaces
$\lambda_n\wh{C}_n$ are immersed in $B_n(\lambda_n\delta_n/2)$,
with constant mean curvature $H_n = H/\lambda_n$ and satisfy
\begin{equation}\label{eqbounded}
\norma{A_{\lambda_n\wh{C}_n}}\leq 2,
\quad \norma{A_{\lambda_n\wh{C}_n}}(\wh{q}_n) = 1.
\end{equation}

This uniform bound on the second fundamental form of $\lambda_n\wh{C}_n$
implies that there is a $\delta > 0$
such that for every $n\in\N$, there exists a connected domain $G_n$ of
$C_n$, containing $q_n$ and such that:

\begin{enumerate}
\item $\lambda_n \phi(G_n) \subset B_n(\delta)$;
\item $\partial[ \lambda_n \phi(G_n)] = \lambda_n \phi(\partial G_n) \subset \partial B_n(\delta)$;
\item $\lambda_n\phi(G_n)$ is embedded and it is a graph over its projection to
$T_{\wh{q}_n}\lambda_n\phi(G_n)$,
with graphing function having uniformly bounded gradient,
for all $n\in\N$.
\end{enumerate}

\begin{figure}
\centering
\includegraphics[width=0.9\textwidth]{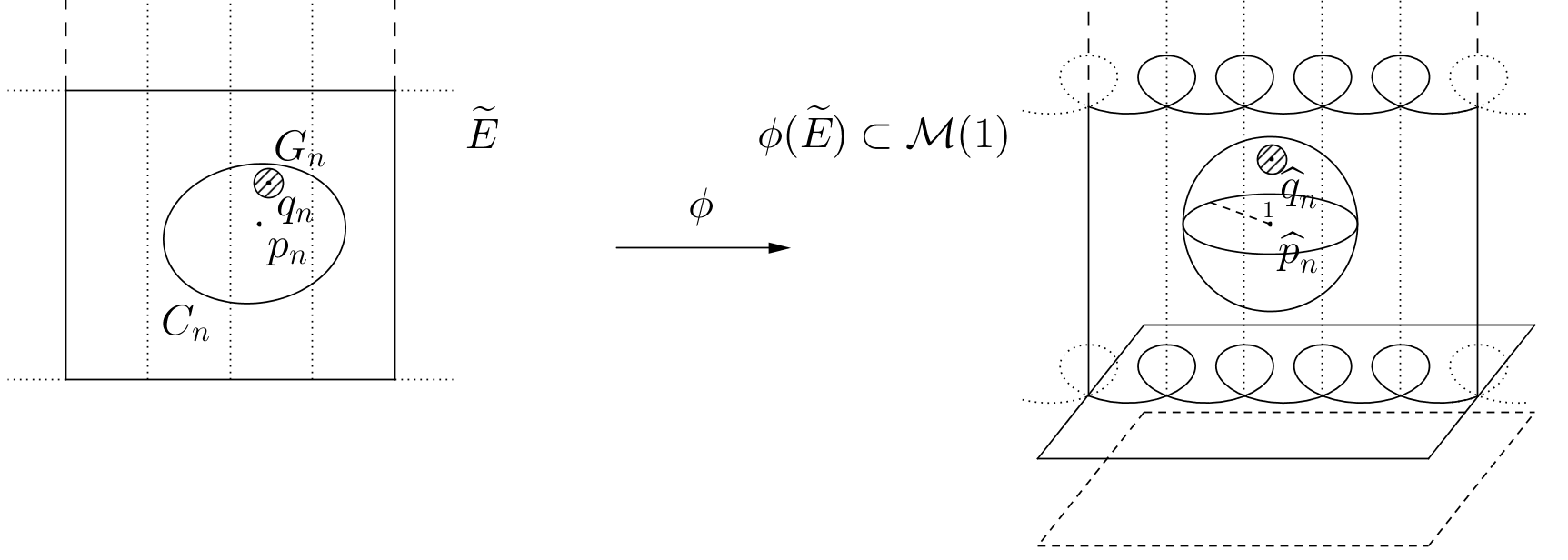}
\caption{The domain $C_n$ containing $p_n$ is immersed in $B_{\hn3}(\wh{p}_n,1)$,
$q_n$ is a maximal point to $f_n$ and the subdomain $G_n \ni q_n$ is
such that $\phi(G_n)$ is an embedded graph over
$T_{\wh{q}_n}\phi(\Et)$.}\label{figCnGn}
\end{figure}

A subsequence of the graphs
$\{\lambda_n\phi(G_n)\}_{n\in\N}$ (which we still denote
by $\{\lambda_n\phi(G_n)\}_{n\in\N}$)
converges, as $n \rightarrow \infty$, to a minimal graph
${G}_\infty\subset \R^3$, embedded in $B_{\R^3}(\vec0,\delta)$ and with $\partial G_\infty
\subset \partial B_{\R^3}(\vec0,\de)$.
We next use the Gauss map $g\colon \Gi\to\sn2$
of $\Gi$ to prove that $E$ has infinite total curvature, from which we
obtain a contradiction.

$\lambda \hn3$ is well-known to be isometric to a Lie group
together with a left invariant metric.
This Lie group is the semidirect product of $\R^2$ with $\R$ having associated
homomorphism $f\colon \R\to \mbox{Gl}(2,\R)$ given by
$f(t)=\exp(tI_\lambda)\in \mbox{Gl}(2,\R)$, where
$I_\lambda$ is the $2\times2$ diagonal matrix
with $1/\lambda$ as diagonal entries; see, for instance,~\cite{mpe11}.
In this $\R^3$-coordinate model for $\lambda\hn3$ the horizontal planes are
left and right cosets of the normal subgroup $\R^2$ of the Lie group,
which we view as being the $(x,y)$-plane with the metric at the origin
$(0,0,0)$ corresponding to the usual metric on $\R^3$;
also left translations by elements in $\R^2$ correspond to parabolic isometries
in the previous upper halfspace model for $\lambda \hn3$
given by translations by horizontal vectors.

From this point of view, for each $n\in\N$,
the metric of $\lambda_n\hn3$ is left
invariant and there is a left invariant Gauss map, defined
for any oriented surface immersed in $\lambda_n\hn3$ and taking values
in the unit sphere $\sn2$ of the tangent space to the identity element
of this semidirect product.
Denote by
$g_n\colon \lambda_n\phi(G_n) \to \sn2$ the left invariant Gauss map
of the oriented embedded surface $\lambda_n\phi(G_n) \subset \lambda_n\hn3$.
Since the group structures of $\lambda_n\hn3$ converge to the abelian group
structure of $\R^3$, then $g_n$ converges to $g\colon G_\infty \to \sn2$,
the Gauss map of $G_\infty$. This convergence is explained
in~\cite{mmpr4}, in the last paragraph of Step~4 of the proof of Theorem~4.1.

Note that \eqref{eqbounded} implies that
$$\norma{A_\Gi}\leq 2,\quad \norma{A_\Gi}(\vec0) = 1.$$

\noindent Hence, $g$ is
injective near $\vec{0}\in \Gi$, since $\Gi$ is a
minimal surface of $\R^3$ with Gaussian curvature
$-1/2$ at $\vec{0}$. It follows that
there exists $\wt{\delta} \in (0,\de)$ such that the graph
$\Gi\cap B_{\R^3}(\vec0,\wt{\de})$ is a disk and
$g\colon\Gi\cap B_{\R^3}(\vec0,\wt{\de})\to \sn2$ is an injective
diffeomorphism with its image.
Note that there exists some $\ve>0$ such that
for all $x \in \Gi\cap B_{\R^3}(\vec0,\wt{\de})$ and $X \in T_x \Gi$
it holds
$\norma{dg_x(X)}\geq \ve\norma{X}$. Then, the fact that $g_n\to g$
in the $C^{1,\alpha}$-topology implies that
there exists a
$\delta^\prime \in (0,\wt{\delta})$ such that,
for $n$ sufficiently large,
$g_n\colon \lambda_n\phi(G_n)\cap B_n(\delta^\prime)\to \sn2$
is also injective.

The fact that $\Gi$ is not flat also implies that
$\Gi\cap B_{\R^3}(\vec0,\de^\prime)$
has strictly negative total curvature; hence there is $K_0 > 0$ such that,
for $n$ sufficiently large, $\lambda_n\phi(G_n)\cap B_n(\delta^\prime)$
has total curvature less than $-K_0^2$. Since total curvature is invariant
under scalings, it follows that
\begin{equation}\label{neg}
\inte{G_n\cap \phi^{-1}(B_n(\delta^\prime))}K_{\Et} < -K_0^2 < 0.
\end{equation}

The assumption
$d_{\cC}(\psi(\wh{p}_n),\psi(\wh{p}_m)) > 2$ if $n\neq m$ implies that,
for $n\neq m$, $\Pi(G_n)\cap \Pi(G_m) = \emptyset$,
therefore $\{\Pi(G_n\cap \phi^{-1}(B_n(\delta^\prime)))\}_{n\in\N}$
is a collection of pairwise disjoint domains of $E$. Furthermore, the assumption
that the restriction of $g_n$ to $B_n(\delta^\prime)$ is injective implies that
$\Pi\vert_{\phi^{-1}(B_n(\delta^\prime))\cap G_n}$ is injective, then each domain
$\Pi(G_n\cap \phi^{-1}(B_n(\delta^\prime)))$ has negative total curvature,
uniformly bounded away from zero by~\eqref{neg}. Hence, we conclude that
the total curvature of $E$ is infinite.

On the other hand, we next prove the total curvature of $E$ is finite. Let
$\{p_n\}_{n\in\N}\subset E$ be a divergent sequence of points. The proofs of
items~\ref{one} and~\ref{two} apply and show that $\lim_{n\to \infty} I_E(p_n) = 0$;
in particular $I_E$ is bounded. Since $E$ has nonpositive Gaussian curvature,
there are geodesic loops $\gamma_n$ based at each $p_n$, and, as explained
in the proof of Theorem~\ref{thmMain}, $\gamma_1$ and
$\gamma_n$ bound a compact annulus $E_n$ such that $\{E_n\}_{n\in\N}$ exhaust
a subend $E^\prime$ of $E$.
A simple argument using Gauss-Bonnet formula shows that $E^\prime$
has finite absolute
total curvature at most $2\pi$, which implies that $E$ has finite total curvature.
This contradiction shows that
$\norma{A_\phi}$ is bounded, finishing the proof of Proposition~\ref{propbound}.
\end{proof}

\subsection{Asymptotics of annular ends of $\varphi(\S)$: proof of item~\ref{six}.}
\label{subAsymp}

Next, we proceed with the proof of the theorem by proving item~\ref{six}, where
we analyze the asymptotic behavior of $\varphi(E)$, where $E$ is an annular end
representative of an end $e$ of $\S$.
Fix a cusp end $\cC = \cM(1)/G(u,v)$ of $N$.
We now describe the standard constant
mean curvature annular ends in $\cC$. Consider
$(k_1,k_2)\in \Z\times\Z \setminus \{(0,0)\}$
such that the greatest common divisor of $k_1$ and $k_2$ equals $1$.
Let $\tau(k_1,\,k_2)$ be the parabolic isometry of $\hn3$
given by~\eqref{taugeneral}
and let again $\psi\colon \cM(1) \to \cC$ denote the universal covering
transformation related to $G(u,v)$ and $\phi(\Et) \subset \cM(1)$ be
an immersed half-plane as
in the proof of Proposition~\ref{propbound} (see Figure~\ref{figDiagram}).

For each $c \in \R$, let $\cP_{0,c}(k_1,k_2)$ be the vertical plane
defined by
$$\cP_{0,c}(k_1,k_2)\hspace{-0.1cm} =\hspace{-0.1cm} \{(x,y,z)
\hspace{-0.08cm}\in\hspace{-0.08cm} \hn3
\hspace{-0.08cm}\mid\hspace{-0.08cm} (k_1u_y+k_2v_y)x-(k_1u_x+k_2v_x)y+c
\hspace{-0.08cm}= \hspace{-0.08cm}0,\ z >0\}.$$

\noindent Then,
$\tau(k_1,\,k_2)$ leaves invariant $\cP_{0,c}(k_1,k_2)$;
hence,
$$\cA_{0,c}(k_1,k_2)=\psi(\cP_{0,c}(k_1,k_2)\cap\cM(1))$$

\noindent is a properly embedded totally
geodesic annulus in $\cC$.
Note that the family
$\{\cA_{0,c}(k_1,k_2)\}_{c\in\R}$ is
periodic in the sense that there exists $l> 0$
such that for any $c\in \R$, $\cA_{0,c}(k_1,k_2)=\cA_{0,c+l}(k_1,k_2)$,
since for each $(m,n) \in \Z\times\Z$,
$$\tau(m,n)(\cP_{0,c}(k_1,k_2))=\cP_{0,c+(k_1n-k_2m)(u_xv_y-u_yv_x)}(k_1,k_2).$$

Finally, for each $c \in \R$, we consider the two families
$\{\cP_{H,c}^+(k_1,k_2)\}_{H\in(0,1)}$ and
$\{\cP_{H,c}^-(k_1,k_2)\}_{H\in(0,1)}$ of equidistant surfaces to
$\cP_{0,c}(k_1,k_2)$,
formed by tilted planes of constant mean curvature $H\in (0,1)$ that
meet $\{z = 0\}$ along the line $\{(k_1u_y+k_2v_y)x-(k_1u_x+k_2v_x)y+c
= 0\}$, see Figure~\ref{figPlanes}. Each $\cP^{\pm}_{H,c}(k_1,k_2)$ is also
invariant by $\tau(k_1,k_2)$, and
so each of the surfaces $\cP^{\pm}_{H,c}(k_1,k_2)\cap\cM(1)$ descends to $\cC$
as a properly embedded annulus $\cA^\pm_{H,c}(k_1,k_2)$ of constant mean curvature
$H$, and the family $\{\cA^\pm_{H,c}(k_1,k_2)\}_{c\in\R}$
is also periodic with respect to $c$.

\begin{figure}
\centering
\includegraphics[width=0.7\textwidth]{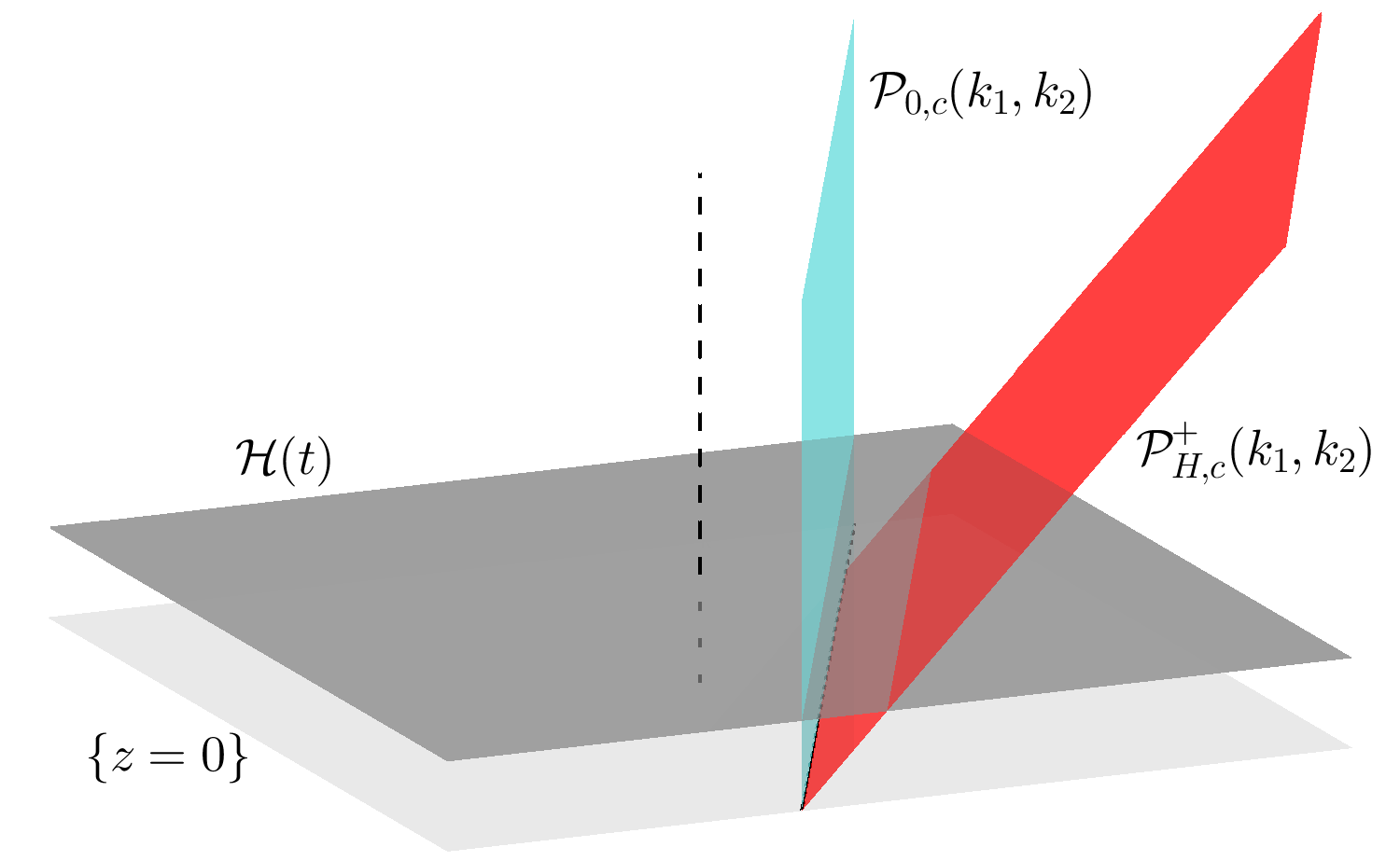}
\caption{The tilted plane $\cP_{H,c}^+(k_1,k_2)$ has constant mean curvature
$H = \cos(\alpha)\in(0,1)$, where $\alpha$ is the angle to the plane $\{z = 0\}$,
and is
equidistant to the totally geodesic vertical plane $\cP_{0,c}(k_1,k_2)$.
\label{figPlanes}}
\end{figure}

We will prove item~\ref{six} by showing that there is
a pair of integers $(k_1,k_2)$ such that $\phi(\Et)$ is
asymptotic either to $\cP_{H,0}^+(k_1,k_2)$ or to $\cP_{H,0}^-(k_1,k_2)$, with
multiplicity $1$.
Furthermore, we will show that $\varphi(E) = \psi(\Et)$ is
asymptotic to $\cA_{H,0}^+(k_1,k_2)$
or to $\cA_{H,0}^{-}(k_1,k_2)$, with some finite
multiplicity $k$.
Since $\cP_{H,c}^+(k_1,k_2)$ is asymptotic to $\cP_{H,0}^+(k_1,k_2)$
(and also $\cP_{H,c}^-(k_1,k_2)$ is asymptotic to $\cP_{H,0}^-(k_1,k_2)$)
for every $c \in \R$, this proves the result.

Let $[\a_1]$ be the generator of $\pi_1(E)$.
Since $\pi_1(\cC)\equiv \Z \times \Z$, we can consider
$\varphi_*([\a_1])$ to be an element of $\Z\times \Z$.
Claim~\ref{claimtransverse} implies that $\varphi_*([\a_1])$ is not
the trivial element, hence there are
relatively prime integers $k_1,\,k_2 \in \Z$ and some $k \in \N$ such that
$\varphi_*([\a_1]) = k(k_1,\,k_2)$. Thus, the map $\tau = \tau(kk_1,kk_2)$
\begin{equation}\label{tau}
\begin{array}{rccl}
\tau\colon& \hn3& \to & \hn3\\
&p&\mapsto & p+k(k_1u+k_2v)
\end{array}
\end{equation}

\noindent is such that $\tau(\phi(\wt{E})) = \phi(\wt{E})$.

Having fixed $k,\,k_1,\,k_2$, we simplify
the notation to $\cP_{H,c}^\pm = \cP_{H,c}^\pm(k_1,k_2)$ and, in
the $c = 0$ case, to $\cP_H^\pm = \cP_{H,0}^\pm(k_1,k_2)$.
We also let
$a = kk_1u_x+kk_2v_x$ and $b = kk_1u_y+kk_2v_y$, so that the
map $\tau$ of \eqref{tau} is
\begin{equation}\label{tauk}
\begin{array}{rccc}
\tau\colon& \hn3& \to & \hn3\\
&(x,y,z)&\mapsto &(x+a,y+b,z),
\end{array}
\end{equation}
\noindent the vertical planes $\cP_{0,c}$ are
$$\cP_{0,c} = \{(x,y,z)\in \hn3\mid bx-ay+c= 0,\,z > 0\}$$

\noindent and the equidistant surfaces $\cP_{H,c}^\pm$ are the tilted planes
with boundary at $\{z = 0\}$ given by the lines $\{bx-ay+c =0\}$.

Since $\partial \phi(\wt{E}) \subset \cH(1)$ is invariant under
the action of $\tau$ and it is a properly immersed curve,
it follows that $\partial \phi(\wt{E})$ stays at a finite distance
to the line $\{(x,y,1)\mid bx-ay = 0\}$. Thus there is some $c > 0 $ such that
$$\partial \phi(\wt{E})\subset \{(x,y,1)\in \hn3\mid \abs{bx-ay} < c\}.$$

\noindent With this, we can prove the following convex hull type property, see
Figure~\ref{figregion}.

\begin{claim}\label{claimregion}$\phi(\wt{E})$ is contained
in the region $\cR$ of $\cM(1)$,
whose boundary contains pieces of all the three planes
$\cP_{H,c}^-$, $\cH(1)$ and $\cP_{H,-c}^+$.
\end{claim}

\begin{proof}[Proof of Claim~\ref{claimregion}]
First, note that $\varphi(E) \subset \cup_{t\geq1} \cT(t)$.
Hence $\phi(\wt{E})$ is contained in $\cM(1)$, and thus it is
never below $\cH(1)$.
Next, we show that $\phi(\wt{E})$ is also never below $\cP_{H,c}^-$.
In the plane $\{z = 0\}$, let $\{C_r\}_{r > 0}$ be a continuous family of
circles of radius $r$, contained in the half plane $\{bx-ay + c < 0\}\cap \{z = 0\}$
and converging,
when $r\rightarrow \infty$, to the line $\{bx-ay + c =0\}\cap\{z = 0\}$.
Let $\cS_r$ be the upper half sphere of radius $r$,
centered in the center of the circle $C_r$. Then,
$\cS_r$ is a totally geodesic surface of $\hn3$ and the family
$\{\cS_r\}_{r> 0}$ converges, when $r\to \infty$,
to the vertical half-plane $\cP_{0,c}$.

\begin{figure}
\centering
\includegraphics[width=0.7\textwidth]{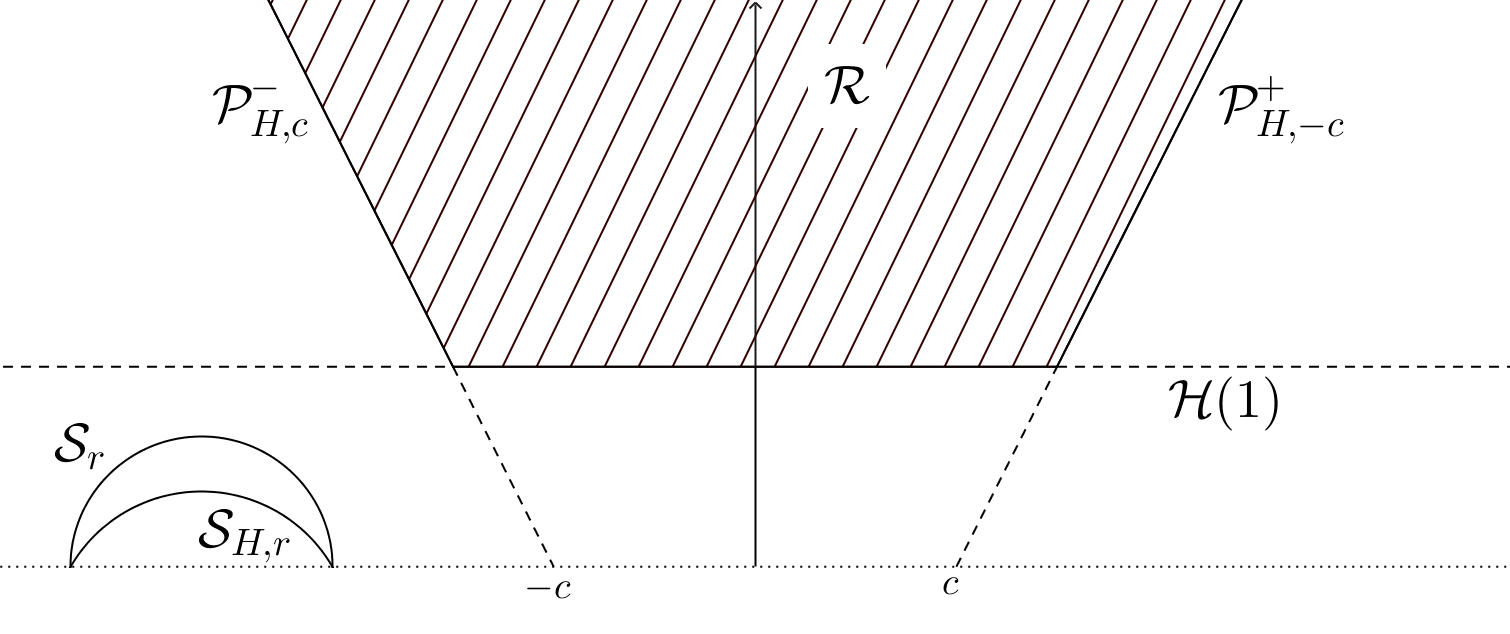}
\caption{The region $\cR$, limited by the three planes
$\cH(1),\,\cP_{H,c}^-$ and $\cP_{H,-c}^+$, and the constant mean curvature $H$
hyperspheres $\cS_{H,r}$.\label{figregion}}
\end{figure}

Let $\cS_{H,r}$
be an equidistant surface to $\cS_r$ with constant mean curvature $H$
pointing upwards,
see Figure~\ref{figregion}.
When $r\rightarrow \infty$,
$\cS_{H,r}$ converges to $\cP_{H,c}^-$.
Each $\cS_{H,r}$ does not intersect $\partial\phi(\wt{E})$,
by its construction. Moreover, for $r$ sufficiently small,
$\cS_{H,r}$ does not intersect $\cM(1)$, so it also
does not intersect $\phi(\wt{E})$. Thus, it
follows from the maximum principle that $\cS_{H,r} \cap \phi(\wt{E}) = \emptyset$
for all $r > 0$, hence there is no point of $\phi(\wt{E})$ below
$\cP_{H,c}^-$. The same argument proves that there is no point of $\phi(\wt{E})$
below $\cP_{H,-c}^+$, proving the claim.
\end{proof}

Consider the family of hyperbolic isometries of $\hn3$ defined, for each $t > 0$, by
\begin{equation}\label{sigmat}
\begin{array}{rccc}
\sigma_t\colon& \hn3 & \to & \hn3\\
& (x,y,z) & \mapsto & e^{-t}(x,y,z).
\end{array}
\end{equation}

\noindent The planes $\cP_0$ and $\cP_{H}^{\pm}$ are invariant
under the action of $\sigma_t$, for all $t> 0$, however, the same does not hold
for $c\neq 0$, since
$\sigma_t(\cP_{0,c}) = \cP_{0,e^{-t}c}$ and
$\sigma_t(\cP_{H,c}^\pm) = \cP_{H,e^{-t}c}^\pm$.
In particular, for every $c$ fixed, $\sigma_t(\cP_{0,c})$
converges to $\cP_0$ and $\sigma_t(\cP_{H,c}^\pm)$ converges to
$\cP_{H}^\pm$ when $t\rightarrow \infty$. Moreover,
$\sigma_t(\cH(1)) = \cH(e^{-t})$.

Let $\cR_t = \sigma_t(\cR)$ and let $\Et_t = \sigma_t(\phi(\Et))$. Then
$\cR_t$ is the
region of $\cM(e^{-t})$ bounded by pieces of
$ \cP_{H,e^{-t}c}^-,\cH(e^{-t})$ and $ \cP_{H,-e^{-t}c}^+$ and $\Et_t \subset \cR_t$.
We also let $\cR_\infty = \lim_{t\rightarrow \infty} \cR_t$ be
the region of $\hn3$ in between the two planes $\cP_H^-$ and $\cP_H^+$.
With this notation, we prove the next claim.

\begin{claim}\label{claimleaves}
Consider the limit set of the family $\{\Et_t\}_{t\geq 0}$,
$\Et_\infty = \{p \in \hn3\mid p = \lim_{n\rightarrow \infty}p_n,\,
p_n\in \Et_{t_n},\, \lim_{n\rightarrow \infty}t_n = \infty\} \subset \cR_\infty$.
Then, for each
$p \in \Et_\infty$ there exists a complete smooth surface $\cL \subset \Et_\infty$,
of constant mean curvature $H$, containing $p$. Moreover, $\cL$ is invariant
under a $1$-parameter group of parabolic isometries which contains
$\tau$.
\end{claim}

\begin{proof}[Proof of Claim~\ref{claimleaves}]
Let $p \in \Et_\infty$.
Since $\wt{E}_t = \sigma_t(\phi(\Et))$ and $\sigma_t$ is an
ambient isometry, it follows that
$\Et_t$ is a constant mean curvature $H$
surface with uniformly bounded norm of the second fundamental form.
Hence there is a $\delta\in (0,1)$ such that for every $t> 0$ and
every $q \in \Et_t$ of distance at least 1 from $\partial \Et_t$,
there exists some closed disk component of $q$ in $\Et_t \cap \overline{B}_{\hn3}(q,\delta)$,
with boundary contained
in $\partial B_{\hn3}(q,\delta)$, which is a graph $G^q$ in exponential coordinates,
over a disk in $T_q\Et_t$, with graphing function having gradient less than 1.
We will also assume that $\de$ is chosen sufficiently small so that,
for $r\in (0,\de]$, each of
the spheres $\partial B_{\hn3}(q,r)$ intersects $G^q$ transversely in a
simple closed curve; now
define $G^q(r)$ be the corresponding closed subdisk of $G^q=G^q(\de)$ in the
closed ball of radius $r\in (0,\de]$
and centered at $q$.

In order to prove the existence of $\cL$ as claimed, consider the sequences
$\{t_n\}_{n\in\N} \subset \R$ and $\{p_n\}_{n\in\N}$ such that
$\lim_{n\to \infty} t_n = \infty$, $\lim_{n\to\infty} p_n = p$
and, for every $n\in\N$, $p_n \in \Et_{t_n}$. For each $n\in\N$ and 
$r\in (0,\de]$, let
$G_n(r)=G^{p_n}(r)\subset \Et_{t_n}$ be the graphs
described above, based at $p_n$; without loss of generality,
we may assume that every $p_n \in \Et_{t_n}$ has distance at
least 1 from $\partial E_t$, and so
the graph $G_n:=G_n(\de)$ exists.
Then $\{G_n\}_{n\in\N}$,
up to a subsequence, converges to a constant mean
curvature $H$ graph $G_\infty\subset \hn3$, with $p \in G_\infty$; by its
construction it follows that $G_\infty \subset \Et_{\infty}$. Since $\delta$ as
above was uniform, we can iterate this argument to extend $G_\infty$
to a complete surface $\cL\subset \Et_\infty$,
containing $p$ and of constant mean curvature $H$.

For $s \in \R$, let
$\tau_{s}\colon \hn3\rightarrow\hn3$ be the parabolic isometry
defined by $$\tau_{s}(x,y,z) = (x+as,y+bs,z);$$ recall that $\tau_1=\tau$
corresponds to the generating covering
transformation $\wh{\tau}$ of $\wt{E}$ and $\tau$ leaves invariant $\phi(\wt{E})$.
Our next argument is to
prove that for all $s \in \R$, $\tau_s(\cL) \subset \cL$, which finishes
the proof of Claim~\ref{claimleaves}.

Since
$\sigma_t\circ \tau = \tau_{e^{-t}}\circ \sigma_t$ and
$\tau(\phi(\Et)) = \phi(\Et)$, it follows that
$$\tau_{e^{-t_n}}(\Et_{t_n})
= \tau_{e^{-t_n}}(\sigma_{t_n}(\phi(\Et)))
= \sigma_{t_n}(\tau(\phi(\Et)))= \Et_{t_n};
$$

\noindent in particular, $\tau_{e^{-t_n}}$ leaves invariant $\Et_{t_n}$.

Consider the graphs $G_n(r)$ as above.
We claim that
there exists an $N_0\in \N$ such that for $n\geq N_0$,
\begin{equation} \label{eq35}
\tau_{e^{-t_n}}(G_n(\delta/3))\subset
G_n(\delta/2).
\end{equation}
Arguing by contradiction, suppose the above statement fails for some
$n$ arbitrarily large,
and so, after replacing by a subsequence, we may assume that
equation~\eqref{eq35} fails for all $n$ sufficiently large.
Note that the isometries $\tau_{e^{-t_n}}$ converge uniformly on
compact subsets of $\hn3$ to the identity map as $n$ approaches infinity.
In particular, for $n$ sufficiently large, we have that
$\tau_{e^{-t_n}}(\overline{B}_{\hn3}(p_n,\delta/3))\subset
{B}_{\hn3}(p_n,\delta/2))$.
However, since $\wt{E}_{t_n}$ is invariant under $\tau_{e^{-t_n}}$,
it follows that $\tau_{e^{-t_n}}(G_n(\delta/3))\subset \wt{E}_{t_n}$; hence,
there exist disks $F_n,\,H_n \subset \wt{E}$ such that
$$\sigma_{t_n}(\phi(F_n)) = G_n(\delta/2),\quad
\sigma_{t_n}( \phi(H_n)) = \tau_{e^{-t_n}}(G_n(\delta/3)).$$
Since $G_n(\delta/2)$ and $\tau_{e^{-t_n}}(G_n(\delta/3))$ have their boundaries in
the disjoint respective spheres
${\partial B}_{\hn3}(p_n,\delta/2))$ and
$\tau_{e^{-t_n}}(\partial\overline{B}_{\hn3}(p_n,\delta/3))$,
then elementary separation properties imply that
$H_n\subset F_n$, $F_n\subset H_n$ or $H_n\cap F_n=\emptyset$.
Note that $H_n\not \subset F_n$ since~\eqref{eq35} is assumed to fail, and
$F_n\not \subset H_n$ since
$\partial F_n$ is clearly disjoint from $H_n$; hence, $H_n\cap F_n=\emptyset$.

The property that $H_n\cap F_n=\emptyset$ gives that the generating covering
transformation $\wh{\tau}$ of the
covering space $\Pi\colon \wt{E}\to E$ induced by $\tau$ is such that
$\wh{\tau}^{-1}(H_n)\subset F_n$
is disjoint from $H_n$. In particular, $\wh{\tau}(H_n)$ is disjoint
from $H_n$, which implies that
$\Pi|_{H_n}$ is injective, for $n$ sufficiently large.
Since $\lim_{n\to \infty} t_n=\infty$, then, after replacing by a subsequence,
we may assume that the projected disks $\Pi(H_n)\subset E$ are pairwise disjoint
for $n\in \N$.
Since the areas of the disks $\Pi(H_n)$ are bounded from below
by some $\ve>0$
and there are an infinite number of them,
we conclude that the area of $E$ is infinite; this contradicts
item~\ref{one} of Theorem~\ref{thmAsymp}, proving that
equation~\eqref{eq35} holds for $n$ sufficiently large.

Next, we prove the claim that $\cL$ is invariant under
the $1$-parameter group of isometries $\{\tau_s\}_{s\in\R}$, that is,
$\tau_s(\cL) = \cL$ for every $s \in \R$. Let
$\Gamma_n = \{\tau_s(p_n)\mid s\in\R\}$ denote
the orbit curve of each $p_n$ through
the action of $\{\tau_s\}_{s\in\R}$. Also letting $\Gamma = \{\tau_s(p)\mid s\in\R\}$
be the orbit curve of $p$, then the curves $\Gamma_n$ converge to $\Gamma$
uniformly as $n\to \infty$.

The arguments used to derive equation~\eqref{eq35} also show
that for any $k\in\N$
there exists an $N_k\in \N$ such that for $n\geq N_k$,
\begin{equation} \label{eq36}
\tau_{e^{-t_n}}(G_n(\delta/(k+1)))\subset G_n(\delta/k).
\end{equation}
After iteration of~\eqref{eq36}, we have that for any positive integer
$j\leq k$, when $n$ is sufficiently large,
$$
\tau_{e^{-jt_n}}(G_n(\delta/(k+1)))\subset G_n. $$
Hence, for any positive integer
$j< k$, $ \tau_{e^{-jt_n}}(p_n) \in G_n$ when $n$ is sufficiently large. Therefore,
as $k$ is arbitrary, $\G_n$ intersects $G_n$ in an arbitrarily large number of points as $n\to \infty$.
Since the analytic arcs $\G_n$ converge smoothly to the analytic arc $\G$
on compact subsets of $\hn3$, the analytic curve $\G$ has infinite order contact with the disk
$G_\infty$, which implies that $\G$ is contained in $\cL$.
Since $p$ was chosen arbitrarily, this proves that the orbit of
any $q\in \cL$, under the action of $\{\tau_s\}_{s\in\R}$ is also contained in $\cL$,
which completes the proof of Claim~\ref{claimleaves}.
\end{proof}

\begin{claim}\label{newclaim}
For any leaf $\cL$ as given in Claim~\ref{claimleaves}, then
either $\cL = \cP_H^+$ or $\cL = \cP_H^-$.
\end{claim}

\begin{proof}[Proof of Claim~\ref{newclaim}]
Any surface invariant under a $1$-parameter group $\mathbb{G}$
of parabolic isometries is
called a \emph{parabolic-invariant surface}, and the intersection of
any such surface with a totally geodesic surface perpendicular to
the orbit curves of $\mathbb{G}$ is called a {\em profile curve}.
The classification of constant mean curvature parabolic-invariant surfaces
is well-explained by Gomes in Chapter~3 of his
doctoral thesis~\cite{gom2}.
It follows from this classification that a constant mean
curvature parabolic invariant
surface contained in $\cR_{\infty}$ is either $\cP_H^+$, $\cP_H^-$ or,
after a homothety, it is a certain surface $\cS$. Furthermore,
such a surface $\cS$ has a specific value of its mean curvature,
is symmetric with respect to the plane $\cP_0$ and attains a maximal height,
which is isolated in any profile curve,
see Figure~\ref{figgraph}.
Our next argument rules out the possibility $\cL = \cS$,
proving the claim.

Arguing by contradiction, suppose $\cL = \cS$. Then
there exists $p \in \cL$ where the height function
of $\cL$ attains its maximal value; in particular $T_p\cL$ is a horizontal plane
in the sense that $T_p\cL = T_p\cH(h_1)$,
for some $h_1>0$.

The uniform bound on the second fundamental form of $\phi(\Et)$ yields a
$\delta >0 $ such that for all $t > 0$ and every point $x \in \Et_t$,
a neighborhood of $\Et_t$ containing $x$ is a graph over a
disk of radius $\delta$ in $T_x\Et_t$. We may assume such $\delta$ is
sufficiently small so that
a neighborhood $G\subset \cL$ containing $p$ is a vertical
graph of small gradient over the horizontal disk $D = B(p,\delta) \subset \cH(h_1)$.
Hence, the height function $f$ of $G$ may be written as
$f\colon D \to \R$, a function over $D$ in such a way that
$$G = \{(x,y,f(x,y))\mid (x,y,h_1) \in D\}.$$

\begin{figure}
\centering
\includegraphics[width = 0.48\textwidth]{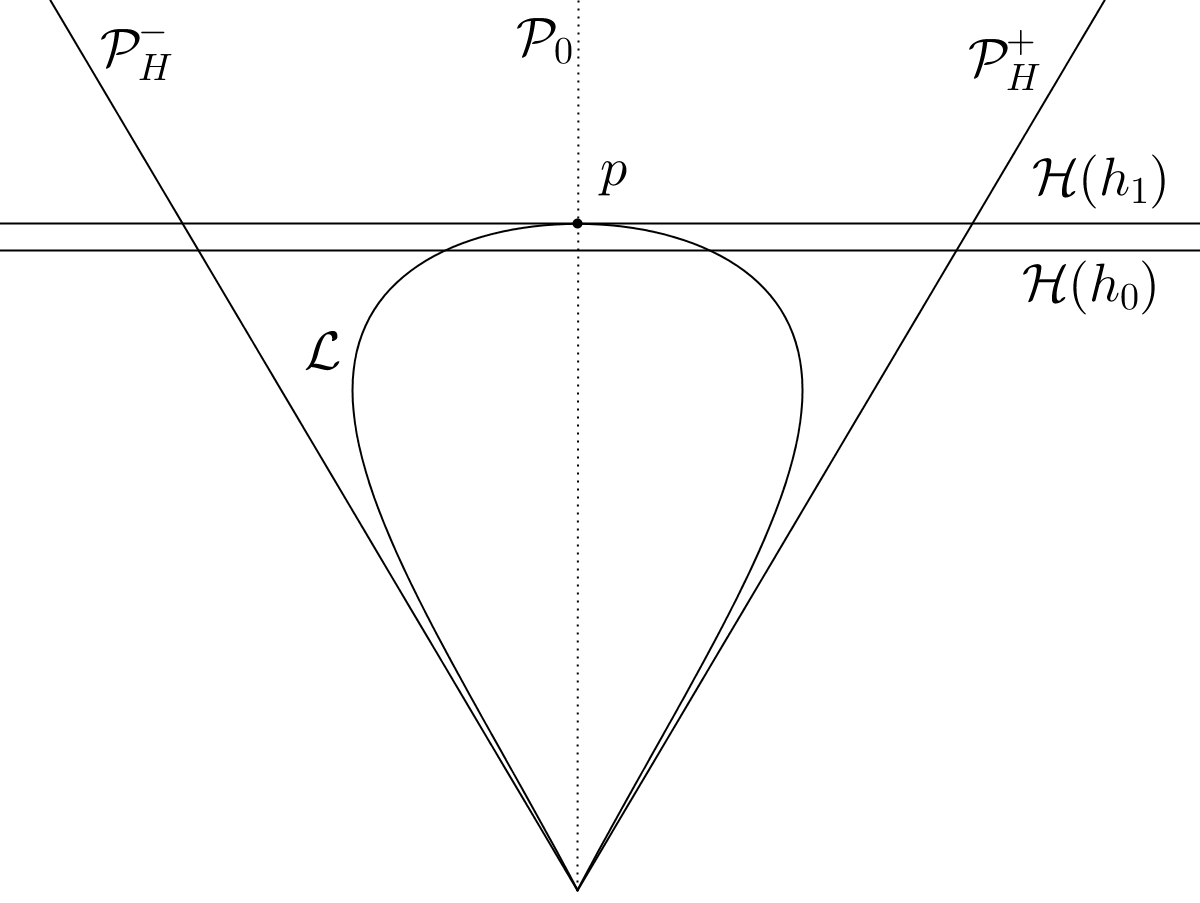}
\includegraphics[width = 0.48\textwidth]{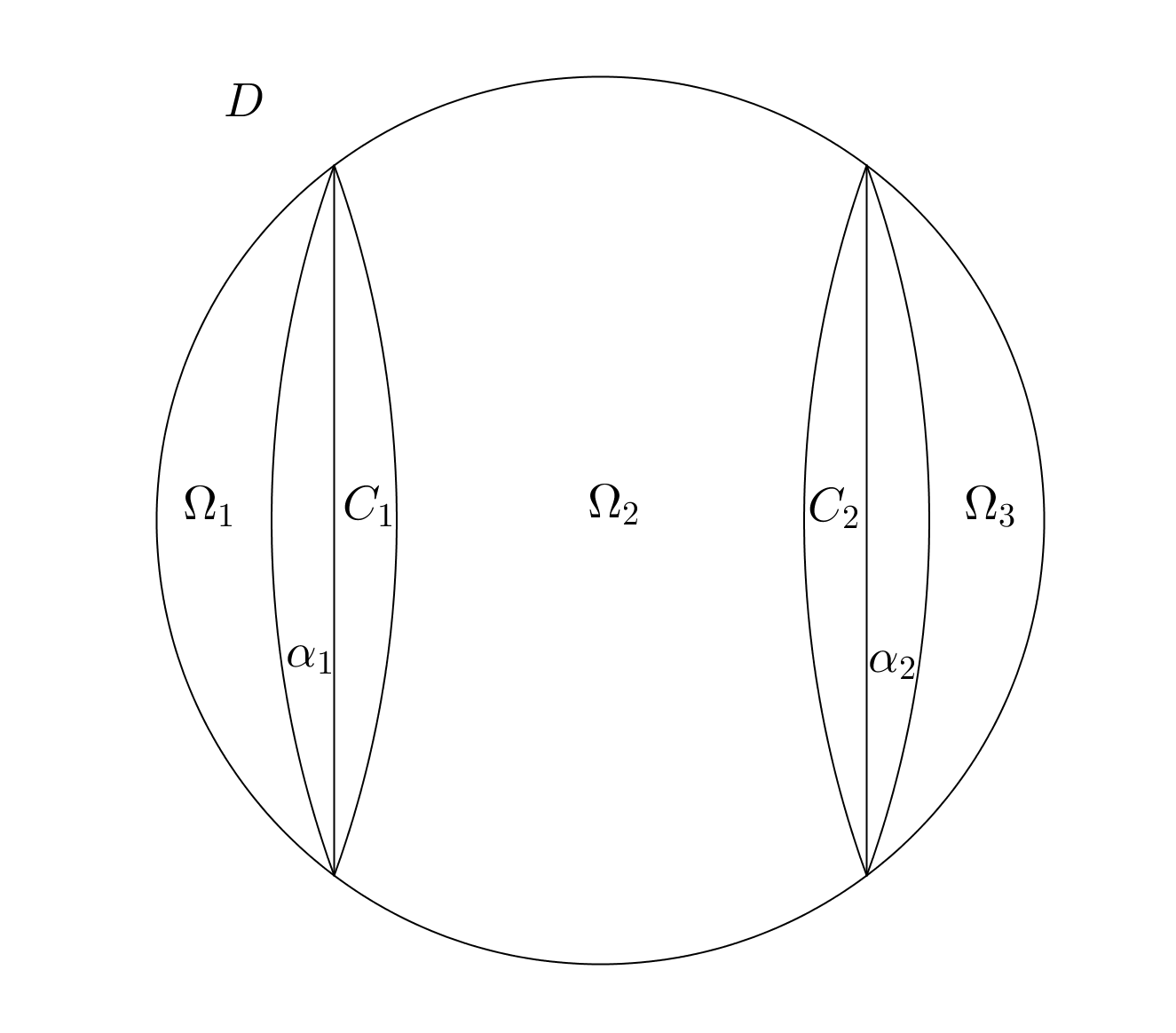}
\caption{If $\cL = \cS$, a neighborhood $G$ of $\cL$ containing $p$ is a horizontal
graph over a domain $D$ in the plane $\cH(h_0)$'.\label{figgraph}}
\end{figure}

Consider a sequence of points $p_n \in \Et_{t_n}$ such that
$\lim_{n\to \infty} p_n = p$ and $\lim_{n\to \infty}t_n = \infty$.
After passing to a subsequence,
we may assume that there are neighborhoods $G_n \subset \Et_{t_n}$
containing $p_n$ which are also vertical graphs over $D$, in the sense that
there exist functions $f_n\colon D \to \R$ such that
$$G_n = \{(x,y,f_n(x,y))\mid (x,y,h_1) \in D\}.$$

\noindent Since $G_n$ converges smoothly to $G$, the sequence $(f_n)_{n\in\N}$
converges in the $C^{2,\alpha}$ norm to $f$.

Let $\Gamma$ be the profile curve of $\cL$ through $p$. Since the
height function of $\Gamma$ attains a strict local maximum at $p$,
there is $h_0 < h_1$ such that $\cH(h_0)$ intersects $\Gamma$ in two points;
in particular $\cH(h_0)$ intersects $G$ in
two disjoint horizontal segments $\alpha_1$, $\alpha_2$,
that separate $G$ into three disjoint regions (see Figure~\ref{figgraph}, right).
Let $\ve = \frac{h_1 - h_0}{2} > 0$, then there exist $n_0 \in \N$ and open,
connected and disjoint domains
$\Omega_1,\,\Omega_2,\,\Omega_3 \subset D$
such that for every $n \geq n_0$ it holds:
$$\left\{
\begin{array}{l}
f_n < h_0- \ve \ \ \text{in}\ \Omega_1 \cup \Omega_3 \\
f_n > h_0 + \ve \ \ \text{in} \ \Omega_2.
\end{array}\right.$$

\noindent In particular, $D \setminus (\Omega_1 \cup \Omega_2 \cup \Omega_3)$
is composed by two disjoint, connected domains $C_1,\,C_2$ such that
$\alpha_i \subset C_i$.

For each $n \geq n_0$, consider $f_n^{-1}(\{h_0\})\subset D$.
It contains two connected curves,
one contained in $C_1$ and the other contained in $C_2$, which map via $f_n$
to two disjoint curves
$\gamma^1_n,\,\gamma^2_n\subset G_n\cap \cH(h_0)$.
Using the same arguments of
the proof of Claim~\ref{claimleaves}, we can fix points $q^1_n \in \gamma^1_n$,
$q^2_n \in \gamma^2_n$ and obtain that
the action of $\tau_{e^{-t_n}}$ is such that both
$\tau_{e^{-t_n}}(q^1_n) \in \gamma^1_n$ and $\tau_{e^{-t_n}}(q^2_n) \in \gamma^2_n$.

Fix any such $n \geq n_0$ and let
$$\begin{array}{c}\Gamma_n^1 = \sigma_{-t_n}(\gamma_n^1)\subset \phi(\Et),\quad
 \Gamma_n^2 = \sigma_{-t_n}(\gamma_n^2)\subset \phi(\Et),\\
\ \\
\wh{q}_n^1 = \sigma_{-t_n}(q_n^1) \in \Gamma_n^1,\quad
\wh{q}_n^2 = \sigma_{-t_n}(q_n^2) \in \Gamma_n^2.\end{array}$$

\noindent Since $\tau_{e^{-t_n}}(q^1_n) \in \gamma^1_n$, it follows that
$\tau(\wh{q}_n^1) \in \Gamma_n^1$, hence $\varphi^{-1}(\Pi(\Gamma_n^1))$ contains a
nontrivial closed curve $\beta_n^1$ in $E$.
Analogously, $\varphi^{-1}(\Pi(\Gamma_n^2))$ also contains
a nontrivial closed curve $\beta_n^2$ in $E$.

We next show that $\beta_n^1\cap\beta_n^2 = \es$ for $n$ sufficiently large;
this proves the claim since $\Gamma_1,\,\Gamma_2 \subset \cH(e^{t_n}h_0)$ gives
$\varphi(\beta_n^1),\,\varphi(\beta_n^2) \subset
\varphi^{-1}\cT(e^{t_n}h_0)$, which is a contradiction
with Claim~\ref{claimtransverse} (since $h_0$ can be chosen generically).
Note that $G(u,v)$, $\sigma_t$ and $\tau_s$ preserve the left invariant Gauss
map (see the discussion in the proof of Proposition~\ref{propbound})
of $\varphi(E)$, $g\colon E \to \sn2$. Moreover, when
$n \to \infty$, $g\vert_{\beta_n^1}$ and $g\vert_{\beta_n^2}$ converge respectively
to the values of the Gauss map of $\cL$ along $\alpha_1$ and $\alpha_2$,
which are distinct.
In particular, for $n$ large enough we have
$\beta_n^1 \cap \beta_n^2 = \es$. As explained before, this
is a contradiction that proves the claim.
\end{proof}

\begin{claim}\label{claimasymp}
There exists $t_* > 1 $ such that $\phi(\Et) \cap \cM(t_*)$ is a topological
half-plane which is a horizontal graph over $\cP_0$.
\end{claim}

\begin{proof}
Fix $\ve>0$.
By Claim~\ref{newclaim}, there exist $t_0$ depending on $\ve$
such that the left invariant Gauss map $g$
of $\phi(\Et) \cap \cM(t_0)$ lies in a $\ve$-neighborhood of $v^+$, $v^-$,
which are the values assumed by the left invariant Gauss maps
of $\cP_H^+$, $\cP_H^-$. We denote such neighborhoods respectively by
$V^+,\,V^-$ and note that, if $\ve$ is sufficiently small,
$V^+\cap V^- = \es$ and both $V^+$, $V^-$ stay at a positive distance to the
great circle in $\sn2$ of vectors perpendicular to the image vector
of the Gauss map of $\cP_0$.

Note
that $t_0$ is such that $\cT(t)$ intersects $\varphi(E)$ transversely for
all $t \geq t_0$.
Let $\alpha \subset E$ be the unique homotopically nontrivial closed curve in
$\varphi^{-1}(\cT(t_0))$ given by Claim~\ref{claimtransverse} and let
$E^\prime\subset E$ be the subannular end of $E$ determined by
$\alpha$, then $\varphi(E^\prime) = \varphi(E) \cap \cup_{t\geq t_0} \cT(t)$,
hence if $\Et^\prime = \Pi^{-1}(E^\prime)$, then it is
topologically a half plane and $\phi^{-1}(\cM(t_0)) = \Et^\prime$.

To finish the proof of the claim, note that
the image $g(\phi(\Et^\prime))$ of the Gauss map of $\phi(\Et^\prime)$
is connected; hence
either $g(\phi(\Et^\prime)) \subset V^+$ or $g(\phi(\Et^\prime))\subset V^-$.
In either case, it follows that $\phi(\Et^\prime) = \phi(\Et)\cap \cM(t_0)$
is a graph over $\cP_0$.
\end{proof}

Note that when $H = 0$, $\Et_{\infty} = \cP_0$ by Claim~\ref{claimleaves}. Hence
Claim~\ref{claimasymp} gives that $\Et_\infty$ contains a single leaf
$\cL = \cP_0$, from which
the asymptotic behavior follows. This special case was proved previously
by Collin, Hauswirth and Rosenberg in \cite{chr2}, Theorem 1.1.

Assume now that $H \in (0,1)$ and that $\phi(\Et)$ is oriented by its
mean curvature vector $\vec{H}$.
To complete the proof of item~\ref{six}, we use the notation
in the proof of Claim~\ref{claimasymp}:
if $g(\phi(\Et)\cap \cM(t_0)) \subset V^+$ (resp. $V^-$),
then $\Et_\infty$ contains
a single leaf $\cL = \cP_H^+$ (resp. $\cP_H^-$).
In either case, the restriction
of $\phi(\Et)$ to $\cM(t_0)$ is a graph over its limit set.
Since $\norma{A_{\phi}}$ is uniformly bounded,
this graphing function converges, with multiplicity $1$,
smoothly to $0$ along $\cH(t)$, when $t\to \infty$.

Note that the covering transformation $\psi\colon \cM(1) \to \cC$
is a finite multiple $k$ of a generator $(k_1,k_2)$ in the fundamental group
$\pi_1(\cC)$. Therefore, $\varphi(E) = \psi(\phi(\Et))$ is asymptotic, with
multiplicity $k$, to the embedded annulus $\cA_H^+ = \psi(\cP_H^+)$
or $\cA_H^- = \psi(\cP_H^-)$.
This completes the proof of Theorem~\ref{thmAsymp}.

\subsection{Some remarks on Theorem~\ref{thmAsymp}.}
A simple consequence of Theorem~\ref{thmAsymp} is the following
corollary, which generalizes to the bounded mean
curvature case some of the corollaries of Theorem~1.1 of~\cite{chr2}.

\begin{corollary}\label{corArea4}
Let $N$ be a complete, noncompact, hyperbolic $3$-manifold of
finite volume. Let $\S$ be a complete, properly immersed surface in $N$ of finite
topology with mean curvature function $H_\S$ satisfying
$\abs{H_\Sigma} \leq H < 1$.
Then, the area of $\Sigma$ satisfies
\begin{equation}\label{are}
\A(\Sigma) \leq \frac{2\pi}{H^2-1}\chi(\S),
\end{equation}

\noindent where $\chi(\S)$ is the Euler characteristic of $\S$ and equality in
\eqref{are} holds if and only if $\Sigma$ is a totally umbilic surface of constant
mean curvature $H$. In particular,
if $\S$ has genus zero, it has at least three ends.

\end{corollary}
\begin{proof}
The proof is straightforward and relies only on item~\ref{three}
of Theorem~\ref{thmAsymp} and on the Gauss equation, after observing
that if $\S$ has genus $g$ and $n$ ends, $\chi(\S) =2-2g-n$.
The details are left to the reader.
\end{proof}

\noindent The next proposition
demonstrates some differences between the $H<1$ case treated
in Theorem~\ref{thmAsymp} and the $H \geq 1$ case.

\begin{proposition}\label{rmksomething}
For any $H \geq 1$ and any hyperbolic $3$-manifold $N$ of finite volume there exists
a complete, properly immersed annulus $\cA$ with constant
mean curvature $H$, and $\cA$ can be chosen to satisfy:

\begin{enumerate}
\item\label{iuno} $\cA$ has infinite area, positive injectivity radius and
bounded norm of its second fundamental form.
\item\label{idos} Any lift of $\cA$ to the hyperbolic $3$-space is a
properly embedded rotationally symmetric annulus.
\end{enumerate}
\end{proposition}

\begin{proof}
After possibly passing to the oriented two-sheeted covering of $N$, we may
assume that $N$ is orientable.
Let $\cC = \cup_{t\geq 1}\cT(t)$ be a parameterized cusp end of $N$ as described in
the proof of Theorem~\ref{thmAsymp}. Let $\Pi\colon \hn3 \to N$ be the universal
covering map of $N$ and assume without loss of generality that
$\Pi(\{z = t\}) = \cT(t)$. Consider
the fundamental group $\Gamma = \pi_1(N) \subset {\rm ISO}(\hn3)$
of $N$ as a subgroup of the isometry group of $\hn3$; each
element $\varphi \in \Gamma$ satisfies $\Pi \circ \varphi = \Pi$.
Since $N$ has finite volume there exists a subgroup $G$ of $\Gamma$ isomorphic
to $\Z \times \Z$ such that $\{z\geq1\}/G \simeq \cC$, and there exists some
$\alpha \in \Gamma \setminus G$.

Let $p \in \partial_\infty \hn3$ be the point of infinity of the horospheres
$\{z = t\}$. Note that $\alpha$ induces a map on $\partial_\infty \hn3$
such that $\alpha(p) \neq p$.
Let $q = \alpha(p)$ and let
$\gamma$ be the complete geodesic of $\hn3$ whose points at infinity are $p$ and $q$.
Let $A$ be an embedded rotationally symmetric annulus around $\gamma$
of constant mean curvature $H$. If $H > 1$, $A$
is a Delaunay surface (see~\cite{hsiang1}), which is periodic and therefore is a
bounded distance from $\gamma$; in particular $\Pi(A)$ is properly immersed
in $N$.
If $H = 1$, then $A$ is called a catenoid cousin (see,
for instance~\cite{br2}) and the intersections $A \cap \{z = t\}$,
$A \cap \alpha(\{z = t\})$ are circles for $t$ sufficiently large, hence
again $\Pi(A)$ is a properly immersed surface in $N$.
In either case, the properties~\ref{iuno} and~\ref{idos} hold.
\end{proof}

\section{Appendix.}\label{secAppendix}
The proof of item~\ref{iD} of Theorem~\ref{thmMain} used the next
elementary intrinsic
result for annular ends of complete surfaces of nonpositive Gaussian curvature.
As we did not find its statement in the literature,
we present its proof in this appendix.

\begin{lemma}\label{lemmaLimitInjRad}
Let $\Sigma$ be a complete surface of finite topology with nonpositive Gaussian
curvature. Let $e$ be an end of $\Sigma$ and $E$ be an annular end
representative. Then for any divergent sequence of points $\{p_n\}_{n\in\N}$ in $E$,

\begin{equation*}
\displaystyle\lim_{n\to \infty} I_\Sigma(p_n)=I^\infty_\Sigma(e)\in [0,\infty].
\end{equation*}
\end{lemma}

\begin{proof}
When $\S$ is simply connected, $I_\S$ is
infinite at every point of $\Sigma$; hence, the lemma holds in this case. Assume now
that $\Sigma$ is not simply connected, thus $I_\S(p)$ is finite for every $p\in\S$.
Let $e$ and $E$ be as in the statement of the lemma and let
$I_E $ denote
the restriction of the injectivity radius function of $\Sigma$ to $E$.

To prove the lemma it suffices to show that
given any two intrinsically divergent sequences of points
$\{p_n\}_{n\in \N},\ \{q_n\}_{n\in \N}$ in $ E$, then the two
limits $\lim_{n\rightarrow \infty}I_E(p_n),\, \lim_{n\rightarrow \infty}I_E(q_n)$
exist in $[0,\infty]$ and
\begin{equation*}
\displaystyle \lim_{n\rightarrow \infty}I_E(p_n) = \lim_{n\rightarrow \infty}I_E(q_n).
\end{equation*}

\noindent The failure of the previous statement
implies that (after possibly passing to subsequences)
there exist intrinsically divergent sequences $\{p_n\}_{n\in\N}$,
$\{q_n\}_{n\in\N}$ in $E$ such that
$$\lim_{n\to\infty} I_E(p_n) = \ell \in [0,\,\infty),\quad
\lim_{n\to\infty} I_E(q_n) = L \in (\ell,\,\infty].$$

\noindent There exist embedded geodesic
loops $\gamma_n$, $\Gamma_n$ based respectively at the points $p_n$, $q_n$ with
$\mbox{\rm Length}(\gamma_n )=2I_E(p_n) = 2\ell_n$, $\mbox{\rm Length}(\Gamma_n) =2I_E(q_n) = 2L_n$,
and such that, for all $n \in \N$,
\begin{equation}\label{limits}
\ell_n< \ell+\ve,\quad L_n > \ell+3\ve,
\end{equation}

\noindent for some $\ve>0$.
After passing to subsequences, we will assume that

\begin{enumerate}
\item the geodesic loops $\g_n$ form a pairwise disjoint family;
\item for all $n,k\in \N$, $p_{n+k}$ lies in the annular subend $W_n$ of $E$ with boundary $\g_n$;
\item $q_n$ lies in the compact annulus in $E$ bounded by $\g_n$ and $\g_{n+1}$.
\end{enumerate}

Since $I_E$ is continuous and $W_n$ is connected, there exists
a point $q_n'\in W_n\setminus W_{n+1}$ with $I(q_n')\in (\ell+3\ve,\ell +4\ve)$.
Hence, replacing the points $q_n$ by the points $q_n^\prime$,
we may assume that $L$ is a finite number.

Let $E_n$ be the compact annulus in $E$ bounded by $\Gamma_n$ and $\Gamma_{n+1}$.
By the same argument as in the proof of Proposition~\ref{propInjPos},
since $\ell$ and $L$ are finite, we may assume that
$\{\gamma_n,\,\Gamma_n\}_{n\in\N}$ is a collection of pairwise disjoint
curves, with $\g_n \subset E_n$.

We claim that there exists a smooth, homotopically nontrivial, simple
closed geodesic $\a_n\subset \Int(E_n)$,
with length at most $\ell_n$.
Let $\Lambda_n $ be the set of simple closed rectifiable curves in $E_n$
homotopic to $\gamma_n$.
If $\beta \in \Lambda_n$ admits a point
$p$ in $\beta \cap B_E(q_n,\ve)$, then
$\len(\beta) > 2\ell_n $. Indeed, if $\len(\beta) \leq 2 \ell_n$, it
would follow from the triangle inequality and from \eqref{limits} that,
for every $x \in \beta$ we have $d_E(q_n,x)\leq L_n$.
Since $B_E(q_n,\,L_n)$
is simply connected, $\beta$ is homotopically trivial in $E$, which implies
that it is also homotopically trivial in $E_n$, contradicting $\beta\in \Lambda_n$.
A similar argument shows that if $\beta$ admits a point $p \in B_E(q_{n+1},\ve)$,
then $\len(\beta)> 2\ell_n$.

Hence, any minimizing sequence in $\Lambda_n$ can be assumed to stay
at least at a distance $\ve$ from the pair of points $q_{n},\,q_{n+1}$, where the boundary of $E_n$
is not smooth.
Then standard minimization arguments imply that
there exists a smooth closed geodesic
$\alpha_n \in \Lambda_n$ which minimizes the lengths of curves in $\Lambda_n$,
and, since $E_n$ is an annulus and $\a_n$ is the generator
of the fundamental group of
$E_n$, $\a_n$ is a simple closed geodesic.

The Gauss-Bonnet formula implies that
each compact annulus $A_n$ bounded by $\alpha_1$ and $\alpha_{n+1}$ is flat,
thus $E$ has a subend $A_\infty$ that is isometric to a flat cylinder with boundary
being a simple closed geodesic.
In fact the flat $A_\infty$ is easily seen to be isometric
to a metric product of a circle with $[0,\infty)$, which implies that
the injectivity radius function has the constant value $\mbox{\rm Length}(\a_1)/2 $
on $A_\infty$. This contradicts \eqref{limits}, proving the lemma.\end{proof}


\center{William H. Meeks, III at profmeeks@gmail.com\\
Mathematics Department, University of Massachusetts, Amherst, MA 01003}

\center{Alvaro K. Ramos, at alvaro.ramos@ufrgs.br\\
Instituto de Matem\'{a}tica e Estat\'{i}stica, Universidade Federal do Rio Grande do
Sul, Porto Alegre, Brazil}

\bibliographystyle{plain}
\bibliography{bill}

\end{document}